\documentclass{amsart}
\usepackage{amsmath, amscd}
\usepackage{stmaryrd}
\usepackage{mathrsfs}
\usepackage{amssymb}
\usepackage{amsfonts}
\usepackage{tensor}
\usepackage[all]{xy}
\usepackage{enumerate}
\usepackage{tikz}
\usepackage{adjustbox}
\usepackage{graphicx}
\usetikzlibrary{matrix, arrows}
\usetikzlibrary{positioning}
\usetikzlibrary{decorations.pathreplacing}
\numberwithin{equation}{subsection}
\theoremstyle{plain}
\newtheorem{thm}[subsection]{Theorem}
\newtheorem{prop}[subsection]{Proposition}
\newtheorem{lemma}[subsection]{Lemma}

\theoremstyle{definition}

\theoremstyle{remark}
\newtheorem{rem}[subsection]{Remark}

\newtheorem{final remark}[subsection]{Final Remark}
\makeatletter
\newcommand*\bigcdot{\mathpalette\bigcdot@{.65}}
\newcommand*\bigcdot@[2]{\mathbin{\vcenter{\hbox{\scalebox{#2}{$\m@th#1\bullet$}}}}}
\makeatother
\begin{document}
\title{Hodge-Witt decomposition of relative crystalline cohomology}
\author{Oliver Gregory and Andreas Langer}
\begin{abstract}
For a smooth and proper scheme over an artinian local ring with ordinary reduction over the perfect residue field we prove - under some general assumptions - that the relative de Rham-Witt spectral sequence degenerates and the relative crystalline cohomology, equipped with its display structure arising from the Nygaard complexes, has a Hodge-Witt decomposition into a direct sum of (suitably Tate-Twisted) multiplicative displays. As examples our main results include the cases of abelian schemes, complete intersections, surfaces, varieties of K3 type and some Calabi-Yau $n$-folds. 
\end{abstract}
\address{Laver Building, University of Exeter, Exeter, EX4 4QF, Devon, UK}
\email {o.b.gregory@exeter.ac.uk}
\address{Harrison Building, University of Exeter, Exeter, EX4 4QF, Devon, UK}
\email{a.langer@exeter.ac.uk}
\date{July 13, 2021 \\ This research was supported by EPSRC grant EP/T005351/1}
\maketitle
\pagestyle{myheadings}

\section{Introduction}

We fix an artinian local ring $R$ with perfect residue field $k$ of characteristic $p>0$. Let $X$ be a smooth proper scheme over $\mathrm{Spec}\,R$. Under some general assumptions on $X$, we proved in \cite{LZ07} and \cite{GL20} that the crystalline cohomology $H_{\rm{cris}}^{i}(X/W(R))$ is equipped with the structure of a higher display, with divided Frobenius maps arising from canonical maps on the Nygaard filtration of the relative de Rham-Witt complex $W\Omega_{X/R}^{\bullet}$. Moreover, if the closed fibre $X_{k}$ of $X$ is an ordinary K3 surface, we proved in \cite{LZ19} that the relative de Rham-Witt spectral sequence
\begin{equation}\label{ss}
E_{1}^{i,j}=H^{j}(X,W\Omega_{X/R}^{i})\Rightarrow\mathbb{H}^{i+j}(X,W\Omega_{X/R}^{\bullet})
\end{equation}
degenerates at $E_1$, giving rise to a Hodge-Witt decomposition of $H^{2}_{{\rm cris}}(X/W(R))$, with its display structure, into a direct sum of displays associated to the formal Brauer group, its twisted dual and the \'{e}tale part of the extended Brauer group.

In this paper we extend this result and produce new examples of the Hodge-Witt decomposition of relative crystalline cohomology. In the following, let $X$ be a smooth proper scheme over ${\rm Spec}\,R$ of relative dimension $d$ satisfying the following assumption: 

There exists a compatible system of smooth liftings $X_{n}/{\rm Spec}\,W_{n}(R)$, $n\in\mathbb{N}$, of $X$ such that the following properties hold:
\\
\begin{itemize}
\item[(A1)] The cohomology $H^{j}(X_{n},\Omega_{X_{n}/W_{n}(R)}^{i})$ is a free $W_{n}(R)$-module for each $i,j$.
\item[(A2)] For each $n$, the Hodge-de Rham spectral sequence
\begin{equation*}
E_{1}^{i,j}=H^{j}(X_{n},\Omega_{X_{n}/W_{n}(R)}^{i})\Rightarrow\mathbb{H}^{i+j}(X_{n},\Omega_{X_{n}/W_{n}(R)}^{\bullet})=H_{\rm{dR}}^{i+j}(X_{n}/W_{n}(R))
\end{equation*}
degenerates at $E_1$.
\end{itemize}
\begin{rem}
We make the following remarks about the assumptions (A1) and (A2): \begin{itemize}
\item[-] The conditions (A1) and (A2) are satisfied in the case of abelian schemes by \cite[Proposition 2.5.2]{BBM82}, for K3 surfaces by ``cohomology and base change'' \cite[II. 5]{Mum70} and the simplicity of the Hodge diamond, and for smooth relative complete intersections by \cite[Theorem 1.5]{Del73}.  
\item[-] Suppose that $pR=0$, i.e. $X$ is a purely equicharacteristic deformation of $X_{k}$. If $\dim(X_{n}/W_{n}(R))<p$ then Ogus' generalisation of the Deligne-Illusie $W_{2}$-lifting theorem shows that the assumption that $X_{n}$ has a smooth lifting $X_{n+1}$ over $\mathrm{Spec}\,W_{n+1}(R)$ implies both (A1) and (A2) for $X_{n}$. Indeed, let $f_{n}:X_{n}\rightarrow\mathrm{Spec}\,W_{n}(R)$ be the structure morphism. Then by \cite[Theorem 8.2.6]{Ogu94}, the filtered complex $(Rf_{n\ast}\Omega_{X_{n}/W_{n}(R)}^{\bullet},Rf_{n\ast}\Omega_{X_{n}/W_{n}(R)}^{\geq\bullet})$ is a ``Fontaine-complex'', and hence (A1) and (A2) follow from \cite[Corollary 5.3.7]{Ogu94}.
\item[-] Suppose again that $pR=0$. If $X_{n}$ admits a Frobenius lifting, i.e. a morphism $X_{n}\rightarrow X_{n}$ of $W_{n}(R)$-schemes which lifts the absolute Frobenius on $X_{n}\times_{\mathrm{Spec}\,\mathbb{Z}/p^{n}}\mathrm{Spec}\,\mathbb{Z}/p$, then \cite[Theorem 1.1 and Corollary 2.3]{Nak97} shows that we do not need the restriction on dimension appearing in the previous remark. Note however that the existence of a Frobenius lift is a very strong condition.
\end{itemize}
\end{rem}

We shall also assume that the closed fibre $X_{k}$ of $X$ satisfies the following assumption:
\\
\begin{itemize}
\item[(A3)] $X_{k}$ has a smooth versal deformation space.
\end{itemize}
\

We say that $X$ admits a Hodge-Witt decomposition of $H^{s}_{\rm{cris}}(X/W(R))$ as displays if the relative Hodge-Witt spectral sequence \eqref{ss} degenerates at $E_{1}$ and if there exists a direct sum decomposition
\begin{equation}
H^{s}_{\rm{cris}}(X/W(R))=\bigoplus_{i+j=s}H^{i}(X,W\Omega_{X/R}^{j})
\end{equation}
on which $H^{s}_{\rm{cris}}(X/W(R))$ is equipped with the display structure arising from the Nygaard complexes, and each $H^{i}(X,W\Omega_{X/R}^{j})$ is equipped with the $(-j)$-fold Tate twist of a multiplicative display structure induced by the Frobenius $F$ on $W\Omega_{X/R}^{j}$ such that
\begin{align*}
W(R)\otimes_{F,W(R)} & H^{i}(X,W\Omega_{X/R}^{j})\rightarrow H^{i}(X,W\Omega_{X/R}^{j}) \\
& x\otimes m \hspace{5mm}\mapsto\hspace{5mm}xFm
\end{align*}
is an isomorphism. Note that the crystalline Frobenius on $H^{s}_{\rm{cris}}(X/W(R))$ induces the map $p^{j}F$ on $H^{i}(X,W\Omega_{X/R}^{j})$. In the case $R=k$ a perfect field, the above is the Hodge-Witt or slope decomposition. 

Recall from \cite[Definition 7.2]{BK86} and \cite[Definition 4.12]{IR83} that a complete variety $X$ over a perfect field $k$ of characteristic $p>0$ is called ordinary if $H^{i}(X,B\Omega^{j}_{X/k})=0$ for all $i\geq 0$ and $j>0$, where $B\Omega^{j}_{X/k}=d\Omega_{X/k}^{j-1}$. This coincides with the usual definition of ordinary for abelian varieties \cite[Lemma 6.2]{vdGK03}, and coincides with having height one formal Brauer group for K3 surfaces \cite[Lemma 1.3]{Nyg83}. By \cite[Th\'{e}or\`{e}me 4.13]{IR83}, ordinary varieties admit a slope decomposition. In this article we consider the generalisation of this to families over $R$. We prove the following:
\begin{thm}\label{surfaces}
Suppose that $k$ is algebraically closed. Let $X$ be a smooth projective surface over $\mathrm{Spec}\,R$ with geometrically connected fibres, and admitting a compatible system of liftings $X_{n}/{\rm Spec}\,W_{n}(R)$ satisfying (A1), (A2), (A3) and such that the closed fibre $X_{k}$ is ordinary. Let $p\geq 3$. Then $X$ admits a Hodge-Witt decomposition of $H^{s}_{\rm{cris}}(X/W(R))$ for all $0\leq s\leq 4$.
\end{thm}

\begin{rem}
For $s=1$ the decomposition coincides with the decomposition of the display associated to the $p$-divisible group $\mathrm{Pic}(X)(p)$ over $R$ of the reduced Picard scheme into connected and \'{e}tale part (see \eqref{direct sum bbm} and the subsequent discussion).  The case $s=3$ is obtained by Poincar\'{e}/Cartier duality and yields the decomposition of the display associated to the $p$-divisible group of the Albanese scheme $\mathrm{Alb}_{X}$. For $s=2$ we get the same decomposition as for ordinary K3 surfaces.
\end{rem}

The main new examples, apart from abelian surfaces which are also covered in the next theorem, are smooth complete intersections in a projective space over $R$ of dimension $2$. They satisfy assumptions (A1), (A2) by \cite[Thm 1.5]{Del73} and are generically ordinary \cite[Th\'{e}or\`{e}me 0.1]{Ill90}. It is well known that the deformations of a complete intersection are unobstructed (see e.g. \cite[Theorem 9.4]{Har10}), and hence satisfy (A3).

\begin{thm}\label{abelian}
Let $A$ be an abelian scheme over $\mathrm{Spec}\,R$ with ordinary closed fibre, such that $\dim A=d<p$. Then $A$ admits a Hodge-Witt decomposition of $H^{s}_{\rm{cris}}(A/W(R))$ in all degrees $0\leq s\leq 2d$.
\end{thm}

\begin{rem}
$H^{s}_{\rm{cris}}(A/W(R))$ is equipped with the exterior power structure $\bigwedge^{s}H^{1}_{\rm{cris}}(A/W(R))$ of the display $H^{1}_{\rm{cris}}(A/W(R))$, which is a direct sum
\begin{equation*}
H^{1}_{\rm{cris}}(A/W(R))=H^{1}(A,W\mathcal{O}_{A})\oplus H^{0}(A,W\Omega_{A/R}^{1})
\end{equation*}
of displays according to the connected resp. \'{e}tale part of the $p$-divisible group of the Picard scheme. So we shall derive a canonical isomorphism 
\begin{equation*}
\bigwedge^{i}H^{1}(A,W\mathcal{O}_{A})\otimes\bigwedge^{j}H^{0}(A,W\Omega_{A/R}^{1})\xrightarrow{\simeq}H^{i}(A,W\Omega_{A/R}^{j})
\end{equation*}
of multiplicative displays. 
\end{rem}

In the final section we shall prove our most general result on Hodge-Witt decompositions under deformation. By an $n$-fold we mean a smooth and proper scheme of dimension $n$, defined over a field. We consider a smooth and proper family of $n$-folds $f:\mathcal{X}\rightarrow\mathcal{S}$ over a smooth base $\mathcal{S}\rightarrow\mathrm{Spf}\,W(k)$, with $n< p$. Suppose that the following two properties hold:
\\
\begin{itemize}
\item[(B1)] The relative Hodge sheaves $R^{j}f_{\ast}\Omega_{\mathcal{X}/\mathcal{S}}^{i}$ are locally free for each $i,j$.
\item[(B2)] The relative Hodge-de Rham spectral sequence
\begin{equation*}
E_{1}^{i,j}=R^{j}f_{\ast}\Omega_{\mathcal{X}/\mathcal{S}}^{i}\Rightarrow\mathbb{R}^{i+j}f_{\ast}\Omega_{\mathcal{X}/\mathcal{S}}^{\bullet}
\end{equation*}
degenerates at $E_1$.
\end{itemize}
(See Remark \ref{assumptions} for a discussion of (B1) and (B2) and the relationship with (A1) and (A2)). Then we prove:

\begin{thm}\label{general}
Suppose that $\mathrm{Spec}\,k\rightarrow\mathcal{S}$ is a $k$-point of $\mathcal{S}$ such that the fibre $X_{k}:=\mathcal{X}\times_{\mathcal{S}}\mathrm{Spec}\,k$ over this point is ordinary. Then for any artinian local ring $R$ with residue field $k$, if $\mathrm{Spec}\,k\rightarrow\mathcal{S}$ lifts to a morphism $\mathrm{Spec}\,R\rightarrow\mathcal{S}$
\begin{equation*}
\begin{tikzpicture}[descr/.style={fill=white,inner sep=1.5pt}]
        \matrix (m) [
            matrix of math nodes,
            row sep=2.5em,
            column sep=2.5em,
            text height=1.5ex, text depth=0.25ex
        ]
        { \mathrm{Spec}\,R & \mathcal{S}  \\
       \mathrm{Spec}\,k  & \ \\};

        \path[overlay,->, font=\scriptsize] 
        (m-1-1) edge (m-1-2)
        (m-2-1) edge (m-1-2)
        ;
        
        \path[overlay, right hook->, font=\scriptsize]
        (m-2-1) edge (m-1-1)
        ;
                                        
\end{tikzpicture}
\end{equation*}
the deformation $X:=\mathcal{X}\times_{\mathcal{S}}\mathrm{Spec}\,R$ of $X_{k}$ admits a Hodge-Witt decomposition of $H_{\mathrm{cris}}^{s}(X/W(R))$ as displays in all degrees $0\leq s\leq 2n$.
\end{thm}
 
\begin{rem}
Notice that Theorem \ref{general} implies results of the same type as Theorem \ref{surfaces} and Theorem \ref{abelian} whenever $X_{k}$ has a smooth versal deformation space such that the versal family satisfies (B1) and (B2) (see Remark \ref{assumptions}). In particular, Theorem \ref{abelian} follows from Theorem \ref{general} because the formal deformation space of an abelian variety is smooth and the de Rham cohomology of the versal family satisfies (B1) and (B2) by \cite[2.1.1]{Del81a}. Similarly Theorem \ref{surfaces} follows from Theorem \ref{general} in the special case of K3 surfaces (by \cite[\S2]{Del81b}) and smooth complete intersections (by \cite[Theorem 1.5]{Del73}). We do not know of a surface $X/\mathrm{Spec}\, R$ satisfying (A1), (A2) and (A3) whose versal family does not satisfy (B1) and (B2). We have decided to include the separate proofs of Theorem \ref{surfaces} and Theorem
\ref{abelian} because the techniques are different and elucidate different aspects of the theory; for example the connection with formal groups.

\end{rem} 
 
Finally let us highlight some important examples for which the theorem is applicable, focussing first on the case of Calabi-Yau $n$-folds:

\begin{itemize}
\item[--] Let $k$ be a finite field and consider the Fermat Calabi-Yau $n$-fold $X=X_{n+2}^{n}(p)$ given by
\begin{equation*}
X_{1}^{n+2}+X_{2}^{n+2}+\cdots+X_{n+2}^{n+2}=0
\end{equation*}
in $\mathbb{P}_{R}^{n+1}$, such that $p=\mathrm{char} \ k\equiv 1\mod n+2$. In this case $X_{k}$ is ordinary by \cite{Suw93} (see also \cite{Tok96}) and the Artin-Mazur formal group $\Phi_{X_{k}/k}$ in degree $n$ has height $1$ \cite[Theorem 5.1]{vdGK03}. Since $X_k$ is a hypersurface its deformations are unobstructed, so the deformation space $\mathcal{S}$ of $X_{k}$ is smooth and we consider the versal family $f:\mathcal{X}\rightarrow\mathcal{S}$. The conditions (B1) and (B2) are satisfied for families of hypersurfaces (indeed they are satisfied for smooth relative complete intersections more generally) \cite[Theorem 1.5]{Del73}.

\item[--] Another important example is provided by the Dwork pencil of Calabi-Yau $n$-folds. Let $X=[X_{1}:X_{2}:\ldots:X_{n+2}]$ be the homogeneous coordinates of $\mathbb{P}^{n+1}_{R}$. Then the Dwork pencil is the one-parameter family $V_{t}$ of Calabi-Yau hypersurfaces in $\mathbb{P}^{n+1}_{R}$ over $t\in\mathbb{P}_{R}^{1}$ defined by $P_{t}(X)=0$ where 
\begin{equation*}
P_{t}(X)=X_{1}^{n+2}+X_{2}^{5}+\cdots+X_{n+2}^{n+2}-(n+2)tX_{1}X_{2}\cdots X_{n+2}\,.
\end{equation*}
We can consider this family over any ring $R$ such that $n+2\in R^{\ast}$. Let $R$ be artinian local with perfect residue field $k$ such that $\mathrm{char} \ k>n+2$. When we specialise the family to $\mathbb{P}_{k}^{n+1}$ under the base change $R\rightarrow k$, it is shown in \cite[Theorem 2.2]{Yu09} that the Dwork family is generically ordinary. Hence we can choose $t\in R$ such that the variety $P_{t_{0}}(X)=0$ is ordinary, where $t_{0}$ is the image of $t$ in $k$. We may argue in the same way as in first example to see that (B1) and (B2) are satisfied for the versal family of $X_{k}$.
\end{itemize}
\

Note that in general, Calabi-Yau $n$-folds $X_{k}$ in positive characteristic can have obstructed deformations when $n\geq 3$ (see e.g. \cite{Hir99}, \cite{Sch04}). On the other hand, it is shown in \cite{Sch03} and \cite{ESB05} that if $X_{k}$ admits a smooth formal lifting and satisfies appropriate torsion-freeness hypotheses on its crystalline cohomology, then $X_{k}$ has a smooth deformation space (this gives another proof that our examples satisfy (A3)). It is widely expected that an ordinary Calabi-Yau $n$-fold $X_{k}$ admits even a \emph{canonical} formal lifting, and therefore has unobstructed deformations (see e.g. \cite[\S1]{AZ19} for a summary of the state of the art). In particular, we expect that our result applies to any Calabi-Yau $n$-fold $X$ over $\mathrm{Spec}\,R$ whose closed fibre is ordinary. 

The final example that we point out as of particular interest is the case of varieties of K3 type. Recall from \cite[Definition 22]{LZ19} that a smooth and proper scheme $X_{k}$ over $\mathrm{Spec}\,k$ of dimension $2d$ is of K3 type if the lower four rows of the Hodge diamond are of the form
\begin{center}
\begin{tikzpicture}[descr/.style={fill=white,inner sep=1.5pt}]
        \matrix (m) [
            matrix of math nodes,
            row sep=0.2em,
            column sep=0.0em,
            text height=1.5ex, text depth=0.25ex
        ]
        {h^{3,0} & & h^{2,1} & & h^{1,2} & & h^{0,3} & \ \ \ \ \ \ \ \ \ & 0 & & 0 & & 0 & & 0 \\
         & h^{2,0} & & h^{1,1} & & h^{0,2} & & = & & 1 & & h^{1,1} & & 1 & \\
         & & h^{1,0} & & h^{0,1} & & & & & & 0 & & 0 & & \\
         & & & h^{0,0} & & & & & & & & 1 & & & \\
        };

       \end{tikzpicture} 
\end{center}
where $h^{i,j}=\dim_{k}H^{j}(X_{0},\Omega_{X_{0}/k}^{i})$. Moreover, we require that there exists a $\sigma\in H^{0}(X_{k},\Omega_{X_{k}/k}^{2})$ such that $\sigma^{d}\in H^{0}(X_{k},\Omega_{X_{k}/k}^{2d})$ defines an isomorphism $\mathcal{O}_{X_{k}}\xrightarrow{\sim}\Omega^{2d}_{X_{k}/k}$ and a $\rho\in H^{2}(X_{k},\mathcal{O}_{k})$ such that $\rho^{d}$ generates $H^{2d}(X_{k},\mathcal{O}_{X_{k}})$ and the pairing 
\begin{equation*}
H^{1}(X_{k},\Omega_{X_{k}/k}^{1})\times H^{1}(X_{k},\Omega_{X_{k}/k}^{1})\rightarrow k\,,\,\,\omega_{1}\times\omega_{2}\mapsto\int\omega_{1}\omega_{2}\sigma^{d-1}\rho^{d-1}
\end{equation*}
is perfect. Note that $\sigma$ induces an isomorphism $\mathcal{T}_{X_{k}}:=\mathrm{Hom}(\Omega^{1}_{X_{k}/k},\mathcal{O}_{X_{k}})\simeq\Omega^{1}_{X_{k}/k}$. It is clear that K3 surfaces are varieties of K3 type. Further examples can be constructed by considering the Hilbert scheme $X^{[d]}$ of $d$ points on a K3 surface $X$ in characteristic zero. Spreading out $X^{[d]}$ over a scheme $S$ which is flat and of finite type over $\mathrm{Spec}\,\mathbb{Z}$ and reducing modulo $p$ gives a variety of K3 type over the residue field for almost all primes $p$ \cite[pg 484]{LZ19}. In these examples, the odd Betti numbers vanish by \cite[Theorem 0.1]{Got90}.   

Let $X_{k}$ be a variety of K3 type over $\mathrm{Spec}\,k$. Since $H^{0}(X_{k},\mathcal{T}_{X_{k}})=H^{2}(X_{k},\mathcal{T}_{X_{k}})=0$, $X_{k}$ has a universal deformation $f:\mathcal{X}\rightarrow\mathcal{S}$ where $\mathcal{S}=\mathrm{Spf}\,W(k)\llbracket T_{1},\ldots,T_{r}\rrbracket$ and $r=h^{1,1}$. Suppose moreover that the odd Betti numbers of $X_{k}$ vanish. Then $f:\mathcal{X}\rightarrow\mathcal{S}$ satisfies (B1) and (B2) by \cite[II.5]{Mum70} (there is no room for non-zero differentials due to the Betti number condition). We may then use Theorem \ref{general} to conclude that if $X_{k}$ is ordinary and $\dim X_{k}<p$, then the crystalline cohomology (in all degrees) of any deformation of $X_{k}$ over an artinian local ring $R$ with residue field $k$ admits a Hodge-Witt decomposition as displays.
\\
\\
\noindent {\bf{Acknowledgements.}} We are grateful to the referee for many helpful and enlightening comments which improved the paper.

\section{Introduction to higher displays}

In this section we present the main tools developed in \cite{LZ04}, \cite{LZ07}, \cite{LZ19}, \cite{Lan18} and \cite{GL20} to impose a display structure on relative crystalline cohomology. For a smooth scheme $X$ over a ring $R$ on which $p$ is nilpotent, we defined in \cite{LZ04} the relative de Rham-Witt complex $W_{n}\Omega_{X/R}^{\bullet}$ as an initial object in the category of $F$-$V$-procomplexes, in particular it is equipped with operators $F:W_{n}\Omega_{X/R}^{\bullet}\rightarrow W_{n-1}\Omega_{X/R}^{\bullet}$ and $V:W_{n-1}\Omega_{X/R}^{\bullet}\rightarrow W_{n}\Omega_{X/R}^{\bullet}$ extending the Frobenius and Verschiebung on the Witt vector sheaf $W_{n}\mathcal{O}_{X}$, and satisfying the standard relations in Cartier theory. $W_{n}\Omega_{X/R}^{\bullet}$ coincides with Deligne-Illusie's de Rham-Witt complex for $R=k$ a perfect field, and its hypercohomology computes the crystalline cohomology of $X/W_{n}(R)$.

Let $X$ be a proper and smooth scheme over $\mathrm{Spec}\, R$ and let $I_{R,n}=\tensor*[^V]{W_{n-1}(R)}{}$ and $I_{R}=\tensor*[^V]{W(R)}{}$. Then we consider the following variant of $W_{n}\Omega_{X/R}^{\bullet}$, denoted by $N^{r}W_{n}\Omega_{X/R}^{\bullet}$, for $r\geq 0$:
\begin{equation*}
{W_{n-1}\Omega_{X/R}^{0}}_{[F]}\xrightarrow{d}{W_{n-1}\Omega_{X/R}^{1}}_{[F]}\xrightarrow{d}\cdots\xrightarrow{d}{W_{n-1}\Omega_{X/R}^{r-1}}_{[F]}\xrightarrow{dV}W_{n}\Omega_{X/R}^{r}\xrightarrow{d}\cdots\,.
\end{equation*}
This is a complex of $W_{n}(R)$-modules, where ${W_{n-1}\Omega_{X/R}^{i}}_{[F]}$ for $i<r$ denotes $W_{n-1}\Omega_{X/R}^{i}$ considered as a $W_{n}(R)$-module via restriction of scalars along $W_{n}(R)\xrightarrow{F}W_{n-1}(R)$. Let $P_{0}:=\mathbb{H}^{m}(X,W\Omega_{X/R}^{\bullet})=\varprojlim_{n}\mathbb{H}^{m}(X,W_{n}\Omega_{X/R}^{\bullet})$ and $P_{r}:=\varprojlim_{n}\mathbb{H}^{n}(X,N^{r}W_{n}\Omega_{X/R}^{\bullet})$. Then there are maps $F_{r}\,:\,P_{r}\rightarrow P_{0}$ induced by corresponding divided Frobenius maps 
\begin{equation*}
\hat{F}_{r}\,:\,N^{r}W_{n}\Omega_{X/R}^{\bullet}\rightarrow{W_{n-1}\Omega_{X/R}^{\bullet}}
\end{equation*}
defined as the identity in degrees $<r$ and as $p^{i}F$ in degree $r+i$, for $i\geq 0$. The standard relations between $F$, $V$ and $d$ imply that $\hat{F}_{r}$ is well-defined. There are also maps $\hat{\iota}_{r}:N^{r+1}W_{n}\Omega_{X/R}^{\bullet}\rightarrow N^{r}W_{n}\Omega_{X/R}^{\bullet}$, $\hat{\alpha}_{r}:I_{R,n}\otimes N^{r}W_{n}\Omega_{X/R}^{\bullet}\rightarrow N^{r+1}_{n}W\Omega_{X/R}^{\bullet}$
given explicitly in \cite{LZ07} and \cite{GL20} that induce three sets of maps:
\begin{enumerate}
\item $\cdots\rightarrow P_{r+1}\xrightarrow{\iota_{r}}P_{r}\xrightarrow{\iota_{r-1}}\cdots\xrightarrow{\iota_{1}}P_{0}$ a chain of $W(R)$-module homomorphisms,
\item $W(R)$-module homomorphisms $\alpha_{r}:I_{R}\otimes P_{r}\rightarrow P_{r+1}$ ,
\item  Frobenius-linear maps $F_{r}:P_{r}\rightarrow P_{0}$
\end{enumerate}
satisfying the following:
\begin{enumerate}[I)]
\item For $r\geq 1$
\begin{equation*}
\begin{tikzpicture}[descr/.style={fill=white,inner sep=1.5pt}]
        \matrix (m) [
            matrix of math nodes,
            row sep=4em,
            column sep=2em,
            text height=1.5ex, text depth=0.25ex
        ]
        { I_{R}\otimes P_{r} & P_{r+1} \\
        I_{R}\otimes P_{r-1} & P_{r} \\};

        \path[overlay,->, font=\scriptsize] 
        (m-1-1) edge node [above]{$\alpha_{r}$} (m-1-2)
        (m-2-1) edge node [above]{$\alpha_{r-1}$}        (m-2-2)
        (m-1-1) edge node [left]{$\mathrm{id}\otimes\iota_{r-1}$} (m-2-1)
        (m-1-2) edge node [right]{$\iota_{r}$} (m-2-2)
         ;
                        
\end{tikzpicture}
\end{equation*}
commutes and the diagonal map $I_{R}\otimes P_{r}\rightarrow P_{r}$ is multiplication. For $r=0$, the composition $I_{R}\otimes P_{0}\xrightarrow{\alpha_{0}}P_{1}\xrightarrow{\iota_{0}} P_{0}$ is multiplication.

\item For $r\geq 0$, 
\begin{align*}
F_{r+1}\circ\alpha_{r}=\tilde{F}_{r} \,:\,
& I_{R}\otimes P_{r}\rightarrow P_{0} \\
& \tensor*[^V]{\xi}{}\otimes x\mapsto\xi F_{r}x \,.
\end{align*}
\end{enumerate}
The above data define the structure of a predisplay $\mathcal{P}=(P_{i},\iota_{i},\alpha_{i}, F_{i})$ on $H_{\rm{cris}}^{n}(X/W(R))$ \cite[Definition 2.2]{LZ07}. We note that properties I) and II) imply the property III) of a predisplay:
\\
\par
III) $F_{r}(\iota_{r}(y))=pF_{r+1}(y)$
\\
\\
(see \cite[pp 155-156]{LZ07}), i.e. the diagram below is commutative
\begin{equation*}
\begin{tikzpicture}[descr/.style={fill=white,inner sep=1.5pt}]
        \matrix (m) [
            matrix of math nodes,
            row sep=4em,
            column sep=3em,
            text height=1.5ex, text depth=0.25ex
        ]
        { P_{r} & P_{0} \\
        P_{r+1} & P_{0} \\};

        \path[overlay,->, font=\scriptsize] 
        (m-1-1) edge node [above]{$F_{r}$} (m-1-2)
        (m-2-1) edge node [above]{$F_{r+1}$}        (m-2-2)
        (m-2-1) edge node [left]{$\iota_{r}$} (m-1-1)
        (m-2-2) edge node [right]{$p$} (m-1-2)
         ;
                        
\end{tikzpicture}
\end{equation*}
The predisplays appearing in this paper are separated, i.e. the map from $P_{r+1}$ to the fibre product induced by the above diagram is injective \cite[Definition 2.3]{LZ07}. A predisplay is of degree $d$ (or a $d$-predisplay) if the maps $\alpha_{r}$ are surjective for $r\geq d$ \cite[Definition 2.4]{LZ07}.

Now assume that $X$ admits a compatible system of liftings $X_{n}/W_{n}(R)$ satisfying the assumptions (A1) and (A2) in the introduction, and let $\mathcal{F}^{r}\Omega_{X_{n}/W_{n}(R)}^{\bullet}$ be the following filtered version of the de Rham complex $\Omega_{X_{n}/W_{n}(R)}^{\bullet}$:
\begin{equation*}
I_{R,n}\otimes\mathcal{O}_{X_{n}}\xrightarrow{pd}I_{R,n}\otimes\Omega_{X_{n}/W_{n}(R)}^{1}\xrightarrow{pd}\cdots\xrightarrow{pd}I_{R,n}\otimes\Omega_{X_{n}/W_{n}(R)}^{r-1}\xrightarrow{d}\Omega_{X_{n}/W_{n}(R)}^{r}\xrightarrow{d}\cdots\,.
\end{equation*}
As one of the main results in \cite{Lan18}, used in \cite{LZ19} and \cite{GL20}, we recall
\begin{thm}\label{filtered comparison}
For $r<p$ the complexes $\mathcal{F}^{r}\Omega_{X_{n}/W_{n}(R)}^{\bullet}$ and $N^{r}W_{n}\Omega_{X/R}^{\bullet}$ are isomorphic in the derived category of $W_{n}(R)$-modules.
\end{thm}
One might view this theorem as a filtered version of the comparison between de Rham-Witt cohomology and the de Rham cohomology of a lifting.

Next we point out that the complexes $\mathcal{F}^{r}\Omega_{X_{n}/W_{n}(R)}^{\bullet}$ possess -- under the assumptions (A1), (A2) -- very nice properties:  the $E_{1}$-hypercohomology spectral sequence associated to $\mathcal{F}^{r}\Omega_{X_{n}/W_{n}(R)}^{\bullet}$ degenerates at $E_{1}$ (compare \cite{LZ07}, Proposition 3.2 and the properties following Proposition 3.1). The theorem yields a description of the $\mathbb{H}^{m}(X,\mathcal{F}^{r}\Omega_{X_{n}/W_{n}(R)}^{\bullet})$ in terms of de Rham cohomology:
\begin{equation*}
P_{r}\cong I_{R,n}L_{0}\oplus I_{R,n}L_{1}\oplus\cdots\oplus I_{R,n}L_{r-1}\oplus L_{r}\oplus\cdots\oplus L_{m}
\end{equation*}
where $L_{i}:=H^{m-i}(X_{n},\Omega_{X_{n}/W_{n}(R)}^{i})$, indeed we have the following lemma.

\begin{lemma}
Let $X_{n}/\mathrm{Spec}\,W_{n}(R)$, $n\in\mathbb{N}$, be a compatible system of smooth and proper liftings of $X/R$ satisfying (A1) and (A2). Then the $E_{1}$-hypercohomology spectral sequence associated to $\mathcal{F}^{r}\Omega_{X_{n}/W_{n}(R)}^{\bullet}$ degenerates at $E_{1}$ and gives a direct sum decomposition for $i\geq r$
\begin{align*}
\mathbb{H}^{m}(X_{n},\mathcal{F}^{r}\Omega_{X_{n}/W_{n}(R)}^{\bullet})\simeq
& I_{R,n}\otimes H^{m}(X_{n},\mathcal{O}_{X_{n}})\oplus I_{R,n}\otimes H^{m-1}(X_{n},\Omega^{1}_{X_{n}/W_{n}(R)})\oplus\cdots \\
&\oplus I_{R,n}\otimes H^{m-(r-1)}(X_{n},\Omega^{r-1}_{X_{n}/W_{n}(R)})\oplus H^{m-r}(X_{n},\Omega^{r}_{X_{n}/W_{n}(R)}) \\ 
& \oplus \cdots\oplus H^{0}(X_{n},\Omega^{m}_{X_{n}/W_{n}(R)})\,.
\end{align*} 
\end{lemma}
\begin{proof}
Consider the complex $\Omega^{\bullet}_{p,r,X_{n}}$
\begin{equation*}
\mathcal{O}_{X_{n}}\xrightarrow{pd}\Omega^{1}_{X_{n}/W_{n}(R)}\xrightarrow{pd}\cdots\xrightarrow{pd}\Omega^{r-1}_{X_{n}/W_{n}(R)}\xrightarrow{d}\Omega^{r}_{X_{n}/W_{n}(R)}\xrightarrow{d}\cdots\,.
\end{equation*}
By \cite[Propositions 3.1 \& 3.2]{LZ07}, the hypercohomology spectral sequence of this complex degenerates and we have a direct sum decomposition
\begin{equation*}
\mathbb{H}^{m}(
X_{n},\Omega^{\bullet}_{p,r,X_{n}})\simeq H_{\mathrm{dR}}^{m}(X_{n}/W_{n}(R))\simeq\bigoplus_{j\geq 0}H^{m-j}(X_{n},\Omega_{X_{n}/W_{n}(R)}^{j})\,.
\end{equation*}
We have an exact sequence of complexes
\begin{equation}
\mathcal{F}^{r}\Omega_{X_{n}/W_{n}(R)}^{\bullet}\rightarrow\Omega^{\bullet}_{p,r,X_{n}}\rightarrow\sigma_{\leq r-1}\Omega^{\bullet}_{p,r,X}\,.
\end{equation}
The hypercohomology spectral sequence of the complex $\sigma_{\leq r-1}\Omega^{\bullet}_{p,r,X}=\sigma_{\leq r-1}\Omega^{\bullet}_{p,r,X_{1}}$ degenerates and leads again to a direct sum decomposition
\begin{equation*}
\mathbb{H}^{m}(X,\sigma_{\leq r-1}\Omega^{\bullet}_{p,r,X})\simeq\bigoplus_{j=0}^{r-1}H^{m-j}(X,\Omega_{X/R}^{j})\,.
\end{equation*} 
Evidently the above exact sequence remains exact after truncation:
\begin{equation}
\mathcal{F}^{r}\Omega_{X_{n}/W_{n}(R)}^{\geq j}\rightarrow\Omega^{\geq j}_{p,r,X_{n}}\rightarrow\sigma_{\leq r-1}\Omega^{\geq j}_{p,r,X}\,.
\end{equation}
Taking cohomology and using base change for de Rham cohomology, we get an exact sequence
\begin{equation*}
0\rightarrow\mathbb{H}^{m}(X_{n},\mathcal{F}^{r}\Omega_{X_{n}/W_{n}(R)}^{\geq j})\rightarrow\mathbb{H}^{m}(X_{n},\Omega^{\geq j}_{p,r,X_{n}})\rightarrow\mathbb{H}^{m}(X,\sigma_{\leq r-1}\Omega^{\geq j}_{p,r,X})\rightarrow 0\,.
\end{equation*}
We have a canonical map 
\begin{equation*}
I_{R,n}H^{m-j}(X_{n},\Omega_{X_{n}/W_{n}(R)}^{j})\rightarrow H^{m-j}(X_{n},\Omega_{X_{n}/W_{n}(R)}^{j})\rightarrow\mathbb{H}^{m}(X_{n},\Omega^{\geq j}_{p,r,X_{n}})
\end{equation*}
induced by the splitting of $\mathbb{H}^{m}(X_{n},\Omega^{\geq j}_{p,r,X_{n}})\rightarrow H^{m-j}(X_{n},\Omega_{X_{n}/W_{n}(R)}^{j})$ and the composite with $\mathbb{H}^{m}(X_{n},\Omega^{\geq j}_{p,r,X_{n}})\rightarrow \mathbb{H}^{m}(X_{n},\sigma_{\leq r-1}\Omega^{\geq j}_{p,r,X})$ is the zero map. Hence we get an induced map $I_{R,n}H^{m-j}(X_{n},\Omega_{X_{n}/W_{n}(R)}^{j})\rightarrow\mathbb{H}^{m}(X_{n},\mathcal{F}^{r}\Omega_{X_{n}/W_{n}(R)}^{\geq j})$ which splits the surjection $\mathbb{H}^{m}(X_{n},\mathcal{F}^{r}\Omega_{X_{n}/W_{n}(R)}^{\geq j})\rightarrow I_{R,n}H^{m-j}(X_{n},\Omega_{X_{n}/W_{n}(R)}^{j})$. The degeneracy of the hypercohomology spectral sequence associated to $\mathcal{F}^{r}\Omega_{X_{n}/W_{n}(R)}^{\bullet}$ (again by \cite[Propositions 3.1 \& 3.2]{LZ07}) then yields the direct sum decomposition for $\mathbb{H}^{m}(\mathcal{F}^{r}\Omega_{X_{n}/W_{n}(R)}^{\bullet})$. 

\end{proof}

Now pass to the projective limit. One defines Frobenius-linear maps $\Phi_{r}:L_{r}\rightarrow P_{0}$ by $\Phi_{r}:=F_{r}|_{L_{r}}$, where $L_{r}=H^{m-r}(\mathcal{X},\Omega^{r}_{\mathcal{X}/W(R)})$ and $\mathcal{X}=\varinjlim_{n}X_{n}$. It is shown in \cite[Theorem 1.1 a)]{GL20}, \cite[Theorem 5.7]{LZ07}, that for $m<p$, the map
\begin{equation*}
\bigoplus_{i=0}^{m}\Phi_{i}\,:\, P_{0}=\bigoplus_{i=0}^{m}L_{i}\rightarrow P_{0}=\bigoplus_{i=0}^{m}L_{i}
\end{equation*}
is a Frobenius-linear isomorphism, and hence we have
\begin{thm}
Let $m<p$. The predisplay $\mathcal{P}=(P_{i},\iota_{i},\alpha_{i},F_{i})$ is a display; it is isomorphic to a display of degree $m$ given by standard data (see \cite[page 149]{LZ07} and \cite[Appendix]{GL20}). 
\end{thm}

In our results on the Hodge-Witt decomposition, it turns out that the display defined on $H_{\rm{cris}}^{n}(X/W(R))$ is a direct sum of twisted multiplicative displays. We recall the definitions.

A $3n$-display $(P,Q,F,V^{-1})$ as defined in \cite[Definition 1]{Zin02} gives rise to a (pre-)display of degree $1$ with $P_{0}=P$, $P_{1}=Q$, $F_{0}=F$ and $F_{1}=V^{-1}$. (Note that, by definition, $I_{R}P\subset Q$, there exists a direct sum decomposition of $W(R)$-modules $P=L\oplus T$ with $Q=L\oplus I_{R}T$ and $V^{-1}$ is an $\tensor[^F]{}{}$-linear isomorphism). 

In the `degenerate' case $(P,Q,F_{0},F_{1})$ with $Q=I_{R}P$ and $F_{1}$ bijective, we call $(P,Q,F_{0},F_{1})$ a multiplicative display (see \cite[\S6]{Mes07}). A special example is the unit display $P_{0}=W(R)$, $P_{1}=I_{R}$, $F_{0}=\tensor[^F]{}{}$ the Frobenius on $W(R)$ and $F_{1}=\tensor[^{V^{-1}}]{}{}$. A multiplicative display (and hence a unit display) has degree $0$. We can extend a multiplicative display to a display $(P_{i},\iota_{i},\alpha_{i},F_{i})$ by setting 
\[ P_{i}=\begin{cases} 
      P & \text{for }i=0 \\
      I_{R}P & \text{for }i\geq 1\,,
   \end{cases}
\]
\[ \iota_{i}=\begin{cases} 
      I_{R}P\hookrightarrow P & \text{for }i=1 \\
      I_{R}P\xrightarrow{p}I_{R}P & \text{for }i\geq 2\,,
   \end{cases}
\]
\[ \alpha_{i}=\begin{cases} 
      I_{R}P\xrightarrow{\mathrm{id}}I_{R}P & \text{for }i=0 \\
      I_{R}\otimes I_{R}P\rightarrow I_{R}P\,;\, \tensor[^V]{\xi}{}\otimes\tensor[^V]{\mu}{}x\mapsto\tensor[^V]{(\xi\mu)}{}x & \text{for }i\geq 1\,,
   \end{cases}
\]
and
\begin{equation*}
F_{i}=F_{1}\text{ for all }i\geq 1\,.
\end{equation*}
For $i\geq 1$, $\alpha_{i}$ is called Verj\"{u}ngung \cite[pp 155-156]{LZ07}. Since the Verj\"{u}ngung is surjective, an extension of a multiplicative display as above is still of degree $0$.

For a multiplicative display $(P_{0},P_{1},F,F_{1})=\mathcal{P}$ we can define the $(-1)$-fold Tate twist as a $1$-display by the data 
\begin{equation*}
\mathcal{P}(-1)=(P'_{i},\iota'_{i},\alpha'_{i}, F'_{i})
\end{equation*}
where for $i\geq 1$ $P'_{i}=P_{i-1}$, $\iota'_{i}=\iota_{i-1}$, $\alpha'_{i}=\alpha_{i-1}$ and $F'_{i}=F_{i-1}$, $P'_{0}=P_{0}=P'_{1}$, $F'_{0}=pF_{0}$, $\iota'_{0}=\mathrm{id}_{P_{0}}$ and $\alpha'_{0}=I_{R}\otimes P_{0}\rightarrow P_{0}$ is the multiplication map. The underlying data
\begin{equation*}
\mathcal{P}^{{\rm\acute{e}t}}=(P=P_{0},Q=P_{0},pF_{0},F_{0})
\end{equation*}
form an \'{e}tale display (for the definition of \'{e}tale displays see also \cite[\S6]{Mes07}). Note that by definition the map $F'_{2}$ in $\mathcal{P}(-1)$ is bijective. We can iterate the construction to define the $(-n)$-fold Tate twist $\mathcal{P}(-n)$ for any $n\geq 0$. It is a display of degree $n$.

The Hodge-Witt decomposition of $H_{\mathrm{cris}}^{m}(X/W(R))$ yields an alternative description as display given by standard data. Since the Frobenius on the de Rham-Witt complex $W\Omega_{X/R}^{\bullet}$ (the crystalline Frobenius) is defined by $p^{i}F$ on $W\Omega^{i}_{X/R}$ where $F$ is the Frobenius on $W\Omega^{i}_{X/R}$, we will consider the $(-i)$-fold Tate twist of $H^{m-i}(X,W\Omega_{X/R}^{i})$ which turns out to be a multiplicative display. In order to identify the display $H^{m}_{\mathrm{cris}}(X/W(R))$ as a direct sum of displays given by the standard data $H^{m-i}(X,W\Omega_{X/R}^{i})$, we tacitly apply an extension of the $(-i)$-fold Tate twist of $P_{0}=H^{m-i}(X,W\Omega_{X/R}^{i})$ (a display of degree $i$) by `adding' maps $\iota_{j}:P_{j}\xrightarrow{p}P_{j-1}$, $P_{j-1}=I_{R}P_{0}$, $F_{j}:P_{j}\rightarrow P_{0}$ given by $\tilde{F}:\tensor[^V]{\xi}{}x\mapsto \xi Fx$ and $\alpha_{j}$ the Verj\"{u}ngung, for $i+2\leq j\leq m$. In the cases where we prove the Hodge-Witt decomposition we will not explicitly mention this extension again, but only state the Hodge-Witt decomposition as a direct sum decomposition of Tate-twisted multiplicative displays.

\section{Constructing a map from Hodge to Hodge-Witt cohomology}

In this section we draw some consequences of the comparison between the Nygaard complex and the complex $\mathcal{F}^{r}\Omega_{X_{n}/W_{n}(R)}^{\bullet}$ (Theorem \ref{filtered comparison}) under the additional assumptions (A1) and (A2), and construct a map
\begin{equation*}
\alpha_{r}:H^{n-r}(X_{\bigcdot},\Omega_{X_{\bigcdot}/W_{\bigcdot}(R)}^{r})\rightarrow H^{n-r}(X,W_{\bigcdot}\Omega_{X/R}^{r})
\end{equation*}
which for $r<p$ maps $I_{R}H^{n-r}(X_{\bigcdot},\Omega_{X_{\bigcdot}/W_{\bigcdot}(R)}^{r})$ to $VH^{n-r}(X,W_{\bigcdot}\Omega_{X/R}^{r})\subset H^{n-r}(X,W_{\bigcdot}\Omega_{X/R}^{r})$ and in many cases turns out to be an isomorphism which induces the degeneracy of the Hodge-Witt spectral sequence. 

We consider the situation at the beginning of the proof of Theorem \ref{filtered comparison} in \cite{Lan18}: assume there exists a closed embedding $i_{n}:X_{n}\hookrightarrow Z_{n}$ such that $Z_{n}$ is a Witt lift of $Z=Z_{n}\times_{W_{n}(R)}R$. Let $D_{n}$ be the PD-envelope of the embedding $i_{n}$ and let $\mathcal{J}$ be the divided power ideal sheaf. By \cite[Theorem 7.2]{BO78} we have a quasi-isomorphism
\begin{equation*}
\Omega^{\geq r}_{X_{n}/W_{n}(R)}\simeq\left(\mathcal{J}^{[r]}\rightarrow\mathcal{J}^{[r-1]}\Omega^{1}_{D_{n}/W_{n}(R)}\rightarrow\cdots\rightarrow\Omega^{r}_{D_{n}/W_{n}(R)}\rightarrow\cdots\xrightarrow{d}\cdots\right)\,.
\end{equation*}
The Witt lift comes with a canonical map $\mathcal{O}_{Z_{n}}\xrightarrow{\mathfrak{S}}W_{n}\mathcal{O}_{X}$ inducing $\mathcal{O}_{D_{n}}\xrightarrow{\mathfrak{S}}W_{n}\mathcal{O}_{X}$ and hence a map, also denoted by $\mathfrak{S}$:
\begin{equation*}
\mathfrak{S}:\Omega^{\geq r}_{D_{n}/W_{n}(R)}\rightarrow W_{n}\Omega^{\geq r}_{X/R}\,.
\end{equation*}
We consider the composite map in cohomology
\begin{align}\label{alpha r}
\alpha_{r}:H^{n-r}(X_{\bigcdot},\Omega_{X_{\bigcdot}/W_{\bigcdot}(R)}^{r})\rightarrow \nonumber
& \mathbb{H}^{n-r}(X_{\bigcdot},\Omega^{\geq r}_{X_{\bigcdot}/W_{\bigcdot}(R)})\rightarrow\mathbb{H}^{n-r}(D_{\bigcdot},\Omega^{\geq r}_{D_{\bigcdot}/W_{\bigcdot}(R)}) \\
& \rightarrow\mathbb{H}^{n-r}(X,W_{\bigcdot}\Omega^{\geq r}_{X/R})\rightarrow H^{n-r}(X,W_{\bigcdot}\Omega^{r}_{X/R})\,.
\end{align}

In the following we construct a map 
\begin{equation}\label{IR map}
I_{R}H^{n-r}(X_{\bigcdot},\Omega^{r}_{X_{\bigcdot}/W_{\bigcdot}(R)})\rightarrow H^{n-r}(X,W_{\bigcdot}\Omega^{r}_{X/R})
\end{equation}
such that the image is contained in the image of $V$ and which is compatible with $\alpha_{r}$. Recall the following diagram of complexes from \cite[(1.3)]{Lan18} (we omit the subscript $D_{n}$ to keep notation light):
\begin{equation}
\label{big diagram}
\begin{adjustbox}{width=12cm}
\begin{tikzpicture}[descr/.style={fill=white,inner sep=1.5pt}]
        \matrix (m) [
            matrix of math nodes,
            row sep=1.5em,
            column sep=1em,
            text height=1.5ex, text depth=0.25ex
        ]
        {  I_{R,n}\mathcal{O} & \ & \ & \ & \ & \ & \ & \ \\
           I_{R,n}\mathcal{J} & I_{R,n}\Omega^1 & \ & \ & \ & \ & \ & \ \\
           \vdots & \vdots & \ddots & \ & \ & \ & \ & \ \\
           I_{R,n}\mathcal{J}^{[r-3]} & I_{R,n}\mathcal{J}^{[r-2]}\Omega^{1} & \cdots & I_{R,n}\Omega^{r-3} & \ & \ & \ & \ \\
           I_{R,n}\mathcal{J}^{[r-2]} & I_{R,n}\mathcal{J}^{[r-2]}\Omega^1 & \cdots & I_{R,n}\mathcal{J}\Omega^{r-3} & I_{R,n}\Omega^{r-2} & \ & \ & \ \\
          I_{R,n}\mathcal{J}^{[r-1]} & I_{R,n}\mathcal{J}^{[r-2]}\Omega^1 & \cdots & I_{R,n}\mathcal{J}^{[2]}\Omega^{r-3} & I_{R,n}\mathcal{J}\Omega^{r-2} & I_{R,n}\Omega^{r-1} & \ \\
          \mathcal{J}^{[r]} & \mathcal{J}^{[r-1]}\Omega^1 & \cdots & \mathcal{J}^{[3]}\Omega^{r-3} & \mathcal{J}^{[2]}\Omega^{r-2} & \mathcal{J}\Omega^{r-1} & \Omega^{r} & \cdots \\
        };

        \path[overlay,->, font=\scriptsize]
        (m-1-1) edge node [above right] {$pd$} (m-2-2)
        (m-2-2) edge node [above right] {$pd$} (m-3-3)
        (m-3-3) edge node [above right] {$pd$} (m-4-4)
        (m-4-4) edge node [above right] {$pd$} (m-5-5)
        (m-5-5) edge node [above right] {$pd$} (m-6-6)
        (m-6-6) edge node [above right] {$d$} (m-7-7);
        
        \path[overlay,->,font=\scriptsize]
        (m-2-1) edge (m-2-2)
        (m-4-1) edge node [above] {$d$} (m-4-2)
        (m-4-2) edge node [above] {$d$} (m-4-3)
        (m-4-3) edge node [above] {$d$} (m-4-4)
        (m-5-1) edge node [above] {$d$} (m-5-2)
        (m-5-2) edge node [above] {$d$} (m-5-3)
        (m-5-3) edge node [above] {$d$} (m-5-4)
        (m-5-4) edge node [above] {$d$} (m-5-5)
        (m-6-1) edge node [above] {$d$} (m-6-2)
        (m-6-2) edge node [above] {$d$} (m-6-3)
        (m-6-3) edge node [above] {$d$} (m-6-4)
        (m-6-4) edge node [above] {$d$} (m-6-5)
        (m-6-5) edge node [above] {$d$} (m-6-6)
        (m-7-1) edge node [above] {$d$} (m-7-2)
        (m-7-2) edge node [above] {$d$} (m-7-3)
        (m-7-3) edge node [above] {$d$} (m-7-4)
        (m-7-4) edge node [above] {$d$} (m-7-5)
        (m-7-5) edge node [above] {$d$} (m-7-6)
        (m-7-6) edge node [above] {$d$} (m-7-7)
        (m-7-7) edge node [above] {$d$} (m-7-8);       
       
\end{tikzpicture}
\end{adjustbox}
\end{equation}
As explained in the proof of Theorem 1.2 in \cite{Lan18}, the sum of the two lower horizontal sequences is quasi-isomorphic to 
\begin{equation*}
0\rightarrow 0\rightarrow\cdots\rightarrow I_{R,n}\Omega_{X_{n}/W_{n}(R)}^{r-1}\xrightarrow{d}\Omega_{X_{n}/W_{n}(R)}^{r}\xrightarrow{d}\Omega_{X_{n}/W_{n}(R)}^{r+1}\xrightarrow{d}\cdots\,.
\end{equation*}
For any $k\leq r$, the degree-wise sum of the $k+1$ lower horizontal sequences is quasi-isomorphic to the truncated complex 
\begin{equation*}
0\rightarrow 0\rightarrow\cdots\rightarrow I_{R,n}\Omega_{X_{n}/W_{n}(R)}^{r-k}\xrightarrow{pd}\cdots\xrightarrow{pd}I_{R,n}\Omega_{X_{n}/W_{n}(R)}^{r-1}\xrightarrow{d}\Omega_{X_{n}/W_{n}(R)}^{r+1}\xrightarrow{d}\cdots
\end{equation*}
(see \cite[(1.4)]{Lan18}). Summing up all horizontal sequences degree-wise yields a complex $\mathrm{Fil}^{r}\Omega_{D_{n}/W_{n}(R)}^{\bullet}$ which is quasi-isomorphic to $\mathcal{F}^{r}\Omega_{X_{n}/W_{n}(R)}^{\bullet}$.

The splitting of the map
\begin{align*}
\mathbb{H}^{n}(0\rightarrow\cdots\rightarrow I_{R,n}\Omega_{X_{n}/W_{n}(R)}^{r-k}\xrightarrow{pd} & \cdots\xrightarrow{pd}I_{R,n}\Omega_{X_{n}/W_{n}(R)}^{r-1}\xrightarrow{d}\Omega_{X_{n}/W_{n}(R)}^{r+1}\xrightarrow{d}\cdots) \\
& \rightarrow I_{R,n}H^{n-(r-k)}(X_{n},\Omega_{X_{n}/W_{n}(R)}^{r-k})
\end{align*}
induces an isomorphism
\begin{equation*}
I_{R,n}H^{n-(r-k)}(X_{n},\Omega_{X_{n}/W_{n}(R)}^{r-k})\simeq\mathbb{H}^{n}(I_{R,n}\mathcal{J}^{[r-k]}\rightarrow I_{R,n}\mathcal{J}^{[r-k-1]}\Omega^{1}_{D_{n}/W_{n}(R)}\rightarrow\cdots\rightarrow I_{R,n}\Omega_{D_{n}/W_{n}(R)}^{r-k})
\end{equation*}
which only depends on $r-k$, not on $r$ or $k$. The morphism of complexes
\begin{equation*}
\begin{tikzpicture}[descr/.style={fill=white,inner sep=1.5pt}]
        \matrix (m) [
            matrix of math nodes,
            row sep=1.5em,
            column sep=1em,
            text height=1.5ex, text depth=0.25ex
        ]
        {  I_{R,n}\mathcal{J}^{[r-k]} & \cdots & I_{R,n}\Omega_{D_{n}/W_{n}(R)}^{r-k} & \tensor*[^V]{\xi}{}\omega \\
\ & \ & W_{n-1}\Omega_{X/R}^{r-k} & \xi\tensor*[^F]{(\mathfrak{S}\omega)}{} \\
        };

        \path[overlay,->, font=\scriptsize]
        (m-1-1) edge (m-1-2)
        (m-1-2) edge (m-1-3)
        (m-1-3) edge (m-2-3);
        
        \path[overlay,|->, font=\scriptsize]
        (m-1-4) edge (m-2-4);               
       
\end{tikzpicture}
\end{equation*}
yields a canonical map 
\begin{equation}
\label{F alpha}
\tensor*[^F]{\alpha}{_{r-k}}:I_{R}H^{n-(r-k)}(X_{\bigcdot},\Omega_{X_{\bigcdot}/W_{\bigcdot}(R)}^{r-k})\rightarrow H^{n-(r-k)}(X,W_{\bigcdot}\Omega_{X/R}^{r-k})
\end{equation}
which after composing with $V$ gives the map
\begin{equation}\label{restriction}
\alpha_{r-k}:I_{R}H^{n-(r-k)}(X_{\bigcdot},\Omega_{X_{\bigcdot}/W_{\bigcdot}(R)}^{r-k})\rightarrow H^{n-(r-k)}(X,W_{\bigcdot}\Omega_{X/R}^{r-k})
\end{equation}
which is the restriction of $\alpha_{r-k}$ in \eqref{alpha r} and has image in $VH^{n-(r-k)}(X,W_{\bigcdot}\Omega_{X/R}^{r-k})$.

In the absence of a global Witt lift we proceed by simplicial methods using the quasi-isomorphisms of simplicial complexes of sheaves
\begin{equation*}
\mathcal{F}^{r}\Omega_{X_{n}^{\bigcdot}/W_{n}(R)}^{\bullet}\xleftarrow{\sim}\mathrm{Fil}^{r}\Omega^{\bullet}_{D_{n}^{\bigcdot}/W_{n}(R)}\xrightarrow{\sim}N^{r}W_{n}\Omega_{X^{\bigcdot}/R}^{\bullet}
\end{equation*}
at the end of the proof of Theorem 1.2 in \cite{Lan18} to construct the map $\alpha_{r}$ with the desired properties. We omit the details here.

By construction we have a commutative diagram for $s<r$
\begin{equation}
\label{comm square}
\begin{tikzpicture}[descr/.style={fill=white,inner sep=1.5pt}]
        \matrix (m) [
            matrix of math nodes,
            row sep=1.5em,
            column sep=1em,
            text height=1.5ex, text depth=0.25ex
        ]
        {  \mathbb{H}^{n-s}(X,N^{r}W_{\bigcdot}\Omega_{X/R}^{\geq s}) & H^{n-s}(X,W_{\bigcdot}\Omega^{s}_{X/R}) \\
\mathbb{H}^{n-s}(X_{\bigcdot},\mathcal{F}^{r}\Omega_{X_{\bigcdot}/W_{\bigcdot}(R)}^{\geq s}) & I_{R}H^{n-s}(X_{\bigcdot},\Omega_{X_{\bigcdot}/W_{\bigcdot}(R)}^{s}) \\
        };

        \path[overlay,->, font=\scriptsize]
        (m-1-1) edge (m-1-2)
        (m-2-1) edge (m-2-2)
        (m-2-1) edge (m-1-1)
        (m-2-2) edge node[right]{$\tensor*[^F]{\alpha}{_{s}}$} (m-1-2);            
       
\end{tikzpicture}
\end{equation}
where the left vertical arrow is induced by truncating the map $\Sigma:\mathrm{Fil}^{r}\Omega_{D_{n}/W_{n}(R)}^{\bullet}\rightarrow N^{r}W_{n}\Omega_{X/R}^{\bullet}$ constructed in \cite[(1.5)]{Lan18}. More precisely, we know that $\mathcal{F}^{r}\Omega_{X_{n}/W_{n}(R)}^{\geq s}$ is quasi-isomorphic to the degree-wise sum of the $(r-s+1)$ lower horizontal sequences in \eqref{big diagram}, denoted by $\mathrm{Fil}^{r}\Omega^{\geq s}_{D_{n}/W_{n}(R)}$, from which we have a canonical map to 
\begin{equation*}
I_{R,n}\Omega_{D_{n}/W_{n}(R)}^{s}\xrightarrow{pd}I_{R,n}\Omega_{D_{n}/W_{n}(R)}^{s-1}\xrightarrow{pd}\cdots\xrightarrow{pd}I_{R,n}\Omega_{D_{n}/W_{n}(R)}^{r-1}\xrightarrow{d}\Omega_{D_{n}/W_{n}(R)}^{r}\xrightarrow{d}\cdots\,.
\end{equation*}
The restriction of the map $\Sigma$ in \cite[(1.5)]{Lan18} to this complex (see \cite[(1.6.3) and (1.6.4)]{Lan18}) then yields by composition a canonical map $\mathcal{F}^{r}\Omega_{X_{\bigcdot}/W(R)}^{\geq s}\rightarrow N^{r}W_{\bigcdot}\Omega_{X/R}^{\geq s}$ which induces the left vertical arrow in the above diagram \eqref{comm square}.

We consider the following morphisms of complexes $N^{r}W_{\bigcdot}\Omega_{X/R}^{\bullet}\rightarrow W_{\bigcdot}\Omega^{\bullet}_{X/R}$ and $\mathcal{F}^{r}\Omega_{X_{\bigcdot}/W_{\bigcdot}(R)}^{\bullet}\rightarrow\Omega_{X_{\bigcdot}/W_{\bigcdot}(R)}^{\bullet}$ given as follows:
\begin{equation*}
\begin{tikzpicture}[descr/.style={fill=white,inner sep=1.5pt}]
        \matrix (m) [
            matrix of math nodes,
            row sep=2em,
            column sep=1.5em,
            text height=1.5ex, text depth=0.25ex
        ]
        {  W_{\bigcdot}\mathcal{O}_{X} & W_{\bigcdot}\Omega_{X/R}^{1} & \cdots & W_{\bigcdot}\Omega_{X/R}^{r-1} & W_{\bigcdot}\Omega_{X/R}^{r} & W_{\bigcdot}\Omega_{X/R}^{r+1} & \cdots \\
 W_{\bigcdot}\mathcal{O}_{X} & W_{\bigcdot}\Omega_{X/R}^{1} & \cdots & W_{\bigcdot}\Omega_{X/R}^{r-1} & W_{\bigcdot}\Omega_{X/R}^{r} & W_{\bigcdot}\Omega_{X/R}^{r+1} & \cdots \\
        };

        \path[overlay,->, font=\scriptsize]
        (m-1-1) edge node[above]{$d$} (m-1-2)
        (m-1-2) edge node[above]{$d$} (m-1-3)
        (m-1-3) edge node[above]{$d$} (m-1-4)
        (m-1-4) edge node[above]{$dV$} (m-1-5)
        (m-1-5) edge node[above]{$d$} (m-1-6)
        (m-1-6) edge node[above]{$d$} (m-1-7)
        (m-2-1) edge node[above]{$d$} (m-2-2)
        (m-2-2) edge node[above]{$d$} (m-2-3)
        (m-2-3) edge node[above]{$d$} (m-2-4)
        (m-2-4) edge node[above]{$d$} (m-2-5)
        (m-2-5) edge node[above]{$d$} (m-2-6)
        (m-2-6) edge node[above]{$d$} (m-2-7)
        (m-1-1) edge node[right]{$p^{r-1}V$}(m-2-1)
        (m-1-2) edge node[right]{$p^{r-2}V$} (m-2-2)
        (m-1-4) edge node[right]{$V$} (m-2-4)
        (m-1-5) edge node[right]{$=$} (m-2-5)
        (m-1-6) edge node[right]{$=$} (m-2-6);            
       
\end{tikzpicture}
\end{equation*}
(it is a morphism of complexes because $Vd=pdV$) and
\begin{equation*}
\begin{tikzpicture}[descr/.style={fill=white,inner sep=1.5pt}]
        \matrix (m) [
            matrix of math nodes,
            row sep=2em,
            column sep=1em,
            text height=1.5ex, text depth=0.25ex
        ]
        {  I_{R}\mathcal{O}_{X_{\bigcdot}} & I_{R}\Omega_{X_{\bigcdot}/W_{\bigcdot}(R)}^{1} & \cdots & I_{R}\Omega_{X_{\bigcdot}/W_{\bigcdot}(R)}^{r-1} & \Omega_{X_{\bigcdot}/W_{\bigcdot}(R)}^{r} & \Omega_{X_{\bigcdot}/W_{\bigcdot}(R)}^{r+1} & \cdots \\
 \mathcal{O}_{X_{\bigcdot}} & \Omega_{X_{\bigcdot}/W_{\bigcdot}(R)}^{1} & \cdots & \Omega_{X_{\bigcdot}/W_{\bigcdot}(R)}^{r-1} & \Omega_{X_{\bigcdot}/W_{\bigcdot}(R)}^{r} & \Omega_{X_{\bigcdot}/W_{\bigcdot}(R)}^{r+1} & \cdots \\
        };

        \path[overlay,->, font=\scriptsize]
        (m-1-1) edge node[above]{$pd$} (m-1-2)
        (m-1-2) edge node[above]{$pd$} (m-1-3)
        (m-1-3) edge node[above]{$pd$} (m-1-4)
        (m-1-4) edge node[above]{$d$} (m-1-5)
        (m-1-5) edge node[above]{$d$} (m-1-6)
        (m-1-6) edge node[above]{$d$} (m-1-7)
        (m-2-1) edge node[above]{$d$} (m-2-2)
        (m-2-2) edge node[above]{$d$} (m-2-3)
        (m-2-3) edge node[above]{$d$} (m-2-4)
        (m-2-4) edge node[above]{$d$} (m-2-5)
        (m-2-5) edge node[above]{$d$} (m-2-6)
        (m-2-6) edge node[above]{$d$} (m-2-7)
        (m-1-1) edge node[right]{$\times p^{r-1}$}(m-2-1)
        (m-1-2) edge node[right]{$\times p^{r-2}$} (m-2-2)
        (m-1-5) edge node[right]{$=$} (m-2-5)
        (m-1-6) edge node[right]{$=$} (m-2-6);

        \path[overlay, right hook->, font=scriptsize]
        (m-1-4) edge (m-2-4);        
       
\end{tikzpicture}
\end{equation*}
where the vertical map $\times p^{i}$ is the inclusion of $I_{R}\Omega_{X_{\bigcdot}/W_{\bigcdot}(R)}^{j}$ into $\Omega_{X_{\bigcdot}/W_{\bigcdot}(R)}^{j}$ multiplied by $p^{i}$.

The construction of the comparison isomorphism $\mathcal{F}^{r}\Omega_{X_{\bigcdot}/W_{\bigcdot}(R)}^{\bullet}\rightarrow N^{r}W_{\bigcdot}\Omega_{X/R}^{\bullet}$ via the map $\Sigma:\mathrm{Fil}^{r}\Omega_{D_{\bigcdot}/W_{\bigcdot}(R)}^{\bullet}\rightarrow N^{r}W_{\bigcdot}\Omega_{X/R}^{\bullet}$ in \cite[(1.5)]{Lan18} and the comparison between $\Omega_{X_{\bigcdot}/W_{\bigcdot}(R)}^{\bullet}$ and $W_{\bigcdot}\Omega_{X/R}^{\bullet}$ show that the induced homomorphisms on hypercohomology 
\begin{equation*}
\lambda_{1}:\mathbb{H}^{\ast}(X,N^{r}W_{\bigcdot}\Omega^{\bullet}_{X/R})\rightarrow\mathbb{H}^{\ast}(X,W_{\bigcdot}\Omega^{\bullet}_{X/R})=H_{\mathrm{cris}}^{\ast}(X/W_{\bigcdot}(R))
\end{equation*}
and 
\begin{equation*}
\lambda_{2}:\mathbb{H}^{\ast}(X_{\bigcdot},\mathcal{F}^{r}\Omega^{\bullet}_{X_{\bigcdot}/W_{\bigcdot}(R)})\rightarrow\mathbb{H}^{\ast}(X_{\bigcdot},\Omega^{\bullet}_{X_{\bigcdot}/W_{\bigcdot}(R)})=H_{\mathrm{dR}}^{\ast}(X_{\bigcdot}/W_{\bigcdot}(R))
\end{equation*}
agree. Then we shall use the following argument in various proofs by induction later on:

Assume the maps $\alpha_{s}:H^{i}(X_{\bigcdot},\Omega_{X_{\bigcdot}/W_{\bigcdot}(R)}^{s})\rightarrow H^{i}(X,W_{\bigcdot}\Omega_{X/R}^{s})$ are isomorphisms, the map $V$ is injective on $H^{i}(X,W_{\bigcdot}\Omega_{X/R}^{s})$ and hence the maps $\tensor[^F]{\alpha}{_s}$ in \eqref{F alpha} are bijective for $s<r$. Then truncation induces a commutative diagram
\begin{equation}
\label{comm square 2}
\begin{tikzpicture}[descr/.style={fill=white,inner sep=1.5pt}]
        \matrix (m) [
            matrix of math nodes,
            row sep=2.5em,
            column sep=1.5em,
            text height=1.5ex, text depth=0.25ex
        ]
        {  \mathbb{H}^{i}(X_{\bigcdot},\mathcal{F}^{r+1}\Omega_{X_{\bigcdot}/W_{\bigcdot}(R)}^{\geq r}) & \mathbb{H}^{i}(X,N^{r+1}W_{\bigcdot}\Omega^{\geq r}_{X/R}) \\
\mathbb{H}^{i}(X_{\bigcdot},\Omega_{X_{\bigcdot}/W_{\bigcdot}(R)}^{\geq r}) & \mathbb{H}^{i}(X,W_{\bigcdot}\Omega^{\geq r}_{X/R}) \\
        };

        \path[overlay,->, font=\scriptsize]
        (m-1-1) edge node[above]{$\simeq$}(m-1-2)
        (m-2-1) edge node[above]{$\simeq$}(m-2-2)
        (m-1-1) edge node[left]{restriction of $\lambda_{2}$}(m-2-1)
        (m-1-2) edge node[right]{restriction of $\lambda_{1}$} (m-2-2);            
       
\end{tikzpicture}
\end{equation}
with horizontal isomorphisms. Since the hypercohomology spectral sequences of $\mathcal{F}^{r+1}\Omega_{X_{\bigcdot}/W_{\bigcdot}(R)}^{\bullet}$ and $\Omega^{\bullet}_{X_{\bigcdot}/W_{\bigcdot}(R)}$ degenerate, the vertical maps in \eqref{comm square 2} are injective and the cokernels are isomorphic to $H^{i}(X,\Omega_{X/R}^{r})$.
\section{Hodge-Witt cohomology as multiplicative displays}

In this section we derive the proof of Theorem \ref{surfaces}. It relies on the following more general proposition which holds in all examples.

\begin{prop}\label{HW}
Fix a pair $(i,j)$, $0\leq i,j\leq d=\dim X$, where $X$ is any of the schemes in Theorems \ref{surfaces}, \ref{abelian} and \ref{general}.
\begin{enumerate}[(i)]
\item There is an exact sequence induced by the action of $V$ on $W\Omega_{X/R}^{j}$
\begin{equation*}
0\rightarrow H^{i}(X,W\Omega_{X/R}^{j})\xrightarrow{V}H^{i}(X,W\Omega_{X/R}^{j})\rightarrow H^{i}(X,\Omega_{X/R}^{j})\rightarrow 0\,.
\end{equation*}

\item Let $\mathfrak{X}$ be the ind-scheme over the ind-scheme $\mathrm{Spec}\,W_{\bigcdot}(R)$ arising from the compatible family of liftings $X_{n}$. Then there exists a multiplicative display $\mathcal{P}=(P,Q=I_{R}P,F,F_{1})$ over $R$ and a homomorphism $P\xrightarrow{\varsigma}H^{i}(X,W\Omega_{X/R}^{j})$ compatible with the action of Frobenius, where $F$ on the right is induced by the Frobenius on $W\Omega_{X/R}^{j}$, such that $\varsigma(I_{R}P)\subset VH^{i}(X,W\Omega_{X/R}^{j})$ and the induced map
\begin{equation*}
\overline{\varsigma}\,:\, P/I_{R}P\rightarrow H^{i}(X,W\Omega_{X/R}^{j})/\mathrm{im}\,V\cong H^{i}(X,\Omega_{X/R}^{j})
\end{equation*}
is an isomorphism of free $R$-modules.

\item The map $\varsigma$ is an isomorphism.

\end{enumerate}

\begin{rem}
The Frobenius-equivariant map $\varsigma$ in $(ii)$ fits into a commutative diagram
\begin{equation*}
\begin{tikzpicture}[descr/.style={fill=white,inner sep=1.5pt}]
        \matrix (m) [
            matrix of math nodes,
            row sep=4em,
            column sep=2em,
            text height=1.5ex, text depth=0.25ex
        ]
        { P & H^{i}(X,W\Omega_{X/R}^{j}) \\
        Q=I_{R}P & H^{i}(X,W\Omega_{X/R}^{j}) \\};

        \path[overlay,->, font=\scriptsize] 
        (m-1-1) edge node [above]{$\varsigma$} (m-1-2)
        (m-2-1) edge node [above]{$\varsigma$}        (m-2-2)
        (m-2-1) edge node [left]{$F_{1}$} (m-1-1)
        (m-1-2) edge node [right]{$V$} (m-2-2)
         ;
                        
\end{tikzpicture}
\end{equation*}
Indeed, 
\begin{equation*}
V(\varsigma(F_{1}(\tensor*[^V]{\xi}{}x)))=V(\varsigma(\xi Fx))=V\xi F\varsigma(x)=\tensor*[^V]{\xi}{}\varsigma(x)=\varsigma(\tensor*[^V]{\xi}{}x) \,.
\end{equation*}
So $\varsigma(I_{R}P)\subset VH^{i}(X,W\Omega_{X/R}^{j})$ and $\overline{\varsigma}$ is well-defined. Note that $F_{1}$ is a bijection.

In many cases $P$ is isomorphic to $H^{i}(\mathfrak{X},\Omega_{\mathfrak{X}/W(R)}^{j})$.
\end{rem}

In this section we will prove Proposition \ref{HW} for surfaces and the case $i\geq 0$, $j=0$ for abelian schemes.
\end{prop}

\begin{lemma}\label{(iii)}
Assume that properties (i) and (ii) in Proposition \ref{HW} hold for a fixed pair $(i,j)$. Then property (iii) holds.
\end{lemma}
\begin{proof}
The proof is identical to the corresponding section in the proof of \cite[Lemma 47]{LZ19}. For completeness we include the argument. The property $\varsigma(I_{R}P)\subset VH^{i}(X,W\Omega_{X/R}^{j})$ implies $\varsigma(I_{n}P)\subset V^{n}H^{i}(X,W\Omega_{X/R}^{j})$ for $I_{n}:=\tensor*[^{V^{n}}]{W(R)}{}$. Then one has a commutative diagram 
\begin{equation*}
\begin{tikzpicture}[descr/.style={fill=white,inner sep=1.5pt}]
        \matrix (m) [
            matrix of math nodes,
            row sep=4em,
            column sep=2em,
            text height=1.5ex, text depth=0.25ex
        ]
        { I_{n-1}P & V^{n-1}H^{i}(X,W\Omega_{X/R}^{j}) \\
        I_{n}P & V^{n}H^{i}(X,W\Omega_{X/R}^{j}) \\};

        \path[overlay,->, font=\scriptsize] 
        (m-1-1) edge node [above]{$\varsigma$} (m-1-2)
        (m-2-1) edge node [above]{$\varsigma$}        (m-2-2)
        (m-2-1) edge node [left]{$F_{1}$} (m-1-1)
        (m-1-2) edge node [right]{$V$} (m-2-2)
         ;
                        
\end{tikzpicture}
\end{equation*}
where $F_{1}$ is bijective again.

We claim that the maps
\begin{equation}\label{map}
\overline{\varsigma}\,:\,\frac{I_{n}P}{I_{n+1}P}\rightarrow\frac{V^{n}H^{i}(X,W\Omega_{X/R}^{j})}{V^{n+1}H^{i}(X,W\Omega_{X/R}^{j})}
\end{equation}
are surjective. For $n=0$ this is property $(ii)$. Let $m\in H^{i}(X,W\Omega_{X/R}^{j})$. Find, by induction, elements $x\in I_{n-1}P$ and $m_{1}\in H^{i}(X,W\Omega_{X/R}^{j})$ such that $V^{n-1}m=\varsigma(x)+V^{n}m_{1}$. We write $x=F_{1}y$ for $y\in I_{n}P$. Then 
\begin{equation*}
V^{n}m=V\varsigma(F_{1}y)+V^{n+1}m_{1}=\varsigma(y)+V^{n+1}m_{1}
\end{equation*}
which ends the induction and proves the claim.

We know that $H^{i}(X,W\Omega_{X/R}^{j})$ is $V$-adically separated \cite[Lemma 39]{LZ19}. Therefore the surjectivity of \ref{map} implies that $P\xrightarrow{\varsigma}H^{i}(X,W\Omega_{X/R}^{j})$ is surjective and $H^{i}(X,W\Omega_{X/R}^{j})$ is $V$-adically complete. Since $V$ is injective by property $(i)$, $H^{i}(X,W\Omega_{X/R}^{j})$ is a reduced Cartier module.

Now consider \eqref{map} as a homomorphism of $W_{n+1}(R)$-modules. We claim that both sides of \eqref{map} are isomorphic as $W_{n+1}(R)$-modules and are noetherian. Since a surjective endomorphism of noetherian modules is an isomorphism, this implies that \eqref{map} is an isomorphism and therefore $\varsigma$ is an isomorphism as well. It therefore suffices to prove the claim:

Since 
\begin{align*}
F_{1}\,:\, 
& I_{n}P\rightarrow I_{n-1}P \\
& \tensor[^{V^{n}}]{\xi}{}x\mapsto \tensor[^{V^{n-1}}]{\xi}{}Fx
\end{align*}
is a bijection, we get a bijection
\begin{equation*}
F_{1}\,:\, I_{n}P/I_{n+1}P\rightarrow I_{n-1}P/I_{n}P\,.
\end{equation*}
Let ${P/I_{1}P}_{[F^{n}]}$ denote the $W_{n+1}(R)$-module given by restriction of scalars along $F^{n}:W_{n+1}(R)\rightarrow R$. Iterating $F_{1}$, we get an isomorphism of $W_{n+1}(R)$-modules
\begin{equation*}
F_{1}^{n}\,:\, I_{n}P/I_{n+1}P\rightarrow{P/I_{1}P}_{[F^{n}]}\,.
\end{equation*}
Since $R$ is $F$-finite, ${P/I_{1}P}_{[F^{n}]}$ is a noetherian $W_{n+1}(R)$-module.

Now for the reduced Cartier module $M=H^{i}(X,W\Omega_{X/R}^{j})$ we obtain, analogously, the isomorphism
\begin{equation*}
V^{n}:{M/VM}_{[F^{n}]}\xrightarrow{\simeq}V^{n}M/V^{n+1}M\,.
\end{equation*} 
Hence properties $(i)$ and $(ii)$ show the claim.
\end{proof}

\begin{lemma}
Proposition \ref{HW} holds for $j=0$, $i\geq 0$ if $X$ is an abelian scheme or a smooth proper surface with geometrically connected fibres, ordinary closed fibre and $k$ is algebraically closed.
\end{lemma}
\begin{proof}
For $i=j=0$, the exact sequence in question reads 
\begin{equation*}
0\rightarrow W(R)\xrightarrow{V}W(R)\rightarrow R\rightarrow 0
\end{equation*}
and there is nothing to prove. $H^{0}(X,W\mathcal{O}_{X})$ is isomorphic to the multiplicative (unit) display.

Now let $X=A$ be an abelian scheme. Then $H^{1}(A,W\mathcal{O}_{A})$ is the reduced Cartier module of the formal $p$-divisible group $\widehat{\mathrm{Pic}}_{A/R}$ with tangent space $H^{1}(A,\mathcal{O}_{A})$, and we have an exact sequence 
\begin{equation*}
0\rightarrow H^{1}(A,W\mathcal{O}_{A})\xrightarrow{V}H^{1}(A,W\mathcal{O}_{A})\rightarrow H^{1}(A,\mathcal{O}_{A})\rightarrow 0\,.
\end{equation*}
Since the closed fibre is ordinary the associated display is multiplicative, so 
\begin{equation*}
(P,Q,F,F_{1})=(H^{1}(A,W\mathcal{O}_{A}), I_{R}H^{1}(A,W\mathcal{O}_{A}), F, F_{1})\,.
\end{equation*}
The Grothendieck-Messing crystal $\mathbb{D}(\widehat{\mathrm{Pic}}_{A/R})$ evaluated at $W(R)\rightarrow R$ yields
\begin{equation*}
\mathbb{D}(\widehat{\mathrm{Pic}}_{A/R})_{W(R)}=\mathbb{D}(\hat{\mathbb{G}}_{m/R}^{d})_{W(R)}=H^{1}(A,W\mathcal{O}_{A})=H^{1}(\mathcal{A},\mathcal{O}_{\mathcal{A}})\,,
\end{equation*}
where $\mathcal{A}$ is a lifting as ind-scheme over $\mathrm{Spec}W_{\bigcdot}(R)$.

For $i\geq 2$ we have a commutative diagram
\\
\begin{equation*}
\begin{tikzpicture}[descr/.style={fill=white,inner sep=1.5pt}]
        \matrix (m) [
            matrix of math nodes,
            row sep=4em,
            column sep=2em,
            text height=1.5ex, text depth=0.25ex
        ]
        { \displaystyle\bigwedge\limits^{i}H^{1}(A,W\mathcal{O}_{A}) & H^{i}(A,W\mathcal{O}_{A}) & H^{i}(A,\mathcal{O}_{A}) \\
       \displaystyle\bigwedge\limits^{i}H^{1}(\mathcal{A},\mathcal{O}_{\mathcal{A}}) & H^{i}(\mathcal{A},\mathcal{O}_{\mathcal{A}}) & \ \\};

        \path[overlay,->, font=\scriptsize] 
        (m-1-1) edge (m-1-2)
        (m-2-1) edge node [above]{$\simeq$} (m-2-2)
        (m-2-1) edge node [left]{$\simeq$} (m-1-1)
        (m-1-2) edge (m-1-3)
        ;
        
        \path[overlay,->>, font=\scriptsize]
        (m-2-2) edge (m-1-3)
        ;
                        
\end{tikzpicture}
\end{equation*}
The right arrow $H^{i}(\mathcal{A},\mathcal{O}_{\mathcal{A}})\rightarrow H^{i}(A,\mathcal{O}_{A})$ is a surjection because de Rham cohomology and Hodge cohomology commute with base change \cite[Corollary 8.3]{Kat70} and we have a commutative diagram
\begin{equation*}
\begin{tikzpicture}[descr/.style={fill=white,inner sep=1.5pt}]
        \matrix (m) [
            matrix of math nodes,
            row sep=4em,
            column sep=2em,
            text height=1.5ex, text depth=0.25ex
        ]
        { H^{i}_{\mathrm{dR}}(\mathcal{A}/W(R)) & H^{i}(\mathcal{A},\mathcal{O}_{\mathcal{A}}) \\
      H^{i}_{\mathrm{dR}}(A/R)) & H^{i}(A,\mathcal{O}_{A}) \\};
        
        \path[overlay,->,font=\scriptsize]
        (m-1-2) edge (m-2-2);
        
        \path[overlay,->>, font=\scriptsize]
        (m-1-1) edge (m-1-2)
        (m-1-1) edge (m-2-1)
        (m-2-1) edge (m-2-2)
        ;
                        
\end{tikzpicture}\,.
\end{equation*}
Hence the exact sequence
\begin{equation*}
0\rightarrow W\mathcal{O}_{A}\xrightarrow{V}W\mathcal{O}_{A}\rightarrow\mathcal{O}_{A}\rightarrow 0
\end{equation*}
induces exact sequences for all $i$
\begin{equation*}
0\rightarrow H^{i}(A,W\mathcal{O}_{A})\xrightarrow{V}H^{i}(A,W\mathcal{O}_{A})\rightarrow H^{i}(A,\mathcal{O}_{A})\rightarrow 0\,.
\end{equation*}
Now, $P=\bigwedge^{i}H^{1}(A,W\mathcal{O}_{A})$ carries the exterior power structure of the multiplicative display $H^{1}(A,W\mathcal{O}_{A})$ and hence is again multiplicative:
\begin{equation*}
\mathcal{P}=(P,I_{R}P,F,F_{1})
\end{equation*}
where $F$ is defined by $\bigwedge^{i}F$ on $P$. The above diagram shows that the $F$-equivariant map $\varsigma\,:\,P\rightarrow H^{i}(A,W\mathcal{O}_{A})$ induces an isomorphism 
\begin{equation*}
\overline{\varsigma}\,:\,P/I_{R}P\xrightarrow{\simeq}H^{i}(A,W\mathcal{O}_{A})/\mathrm{im} V\cong H^{i}(A,\mathcal{O}_{A})\,.
\end{equation*}
We conclude property $(iii)$ by Lemma \ref{(iii)}. Hence the lemma holds for abelian schemes.

Now let $X$ be smooth projective scheme over $\mathrm{Spec}\,R$ of dimension $2$ satisfying assumptions (A1), (A2) with ordinary closed fibre. Recall that 
\begin{align*}
H^{0}(X_{k},B\Omega^{1}_{X_{k}/k})
& \cong\ker\left(C:H^{0}\left(X_{k},\left(\Omega^{1}_{X_{k}/k}\right)_{d=0}\right)\rightarrow H^{0}(X_{k},\Omega^{1}_{X_{k}/k})\right) \\
& \cong H^{1}((X_{k})_{\mathrm{fppf}},\alpha_{p}) \\
& \cong \mathrm{Hom}_{k-\mathrm{grp}}(\alpha_{p},\mathrm{Pic}_{X_{k}})
\end{align*}
by applying \cite[(2.1.11)]{Ill79}, \cite[Proposition 4.14]{Mil80} and \cite[Proposition 4.16]{Mil80} in turn (recall that the Cartier dual of $\alpha_{p}$ is itself). Since $X_{k}$ is ordinary, we conclude that the above groups are trivial. In particular, $(\mathrm{Pic}^{0}_{X_{k}})_{\mathrm{red}}$ is ordinary.  The formal Picard group $\widehat{Pic}_{X_{k}/k}$ of $X_{k}$ is the formal completion of $(\mathrm{Pic}^{0}_{X_{k}})_{\mathrm{red}}$ along the zero section \cite[Remark (1.9)(ii)]{AM77}, so we conclude that $\widehat{Pic}_{X_{k}/k}$ is multiplicative. Since $k$ is algebraically closed, we therefore have $\widehat{Pic}_{X_{k}/k}\simeq\hat{\mathbb{G}}_{m}^{g}$. By rigidity of $\hat{\mathbb{G}}_{m}$ we see that $\widehat{Pic}_{X/R}\simeq\hat{\mathbb{G}}_{m/R}^{g}$ and the case $i=1$ is covered by the same argument as in the above discussion.

By \cite[Corollary 3.3]{AM77} the Cartier module of the formal Brauer group $\widehat{\mathrm{Br}}_{X_{k}/k}$ is $H^{2}(X_{k},W\mathcal{O}_{X_{k}})$. Under the assumptions (A1) and (A2), the crystalline cohomology of $X_{k}$ is torsion-free. Since $X_{k}$ is ordinary, the Newton and Hodge polygons of $X_{k}$ coincide \cite[Proposition 7.3]{BK86}. Since $H^{2}(X_{k},W\mathcal{O}_{X_{k}})$ is the slope $0$ part of the crystalline cohomology, its rank (i.e. the height of $\widehat{\mathrm{Br}}_{X_{k}/k}$) equals the dimension of $H^{2}(X,\mathcal{O}_{X})$, and hence $\widehat{\mathrm{Br}}_{X_{k}/k}$ is multiplicative. Since $k$ is algebraically closed, we see that $\widehat{\mathrm{Br}}_{X_{k}/k}\simeq\hat{\mathbb{G}}_{m}^{h}$, and therefore $\widehat{\mathrm{Br}}_{X/R}\simeq\hat{\mathbb{G}}_{m/R}^{h}$ by rigidity. Then it is clear that 
\begin{equation*}
0\rightarrow H^{2}(X,W\mathcal{O}_{X})\xrightarrow{V}H^{2}(X,W\mathcal{O}_{X})\rightarrow H^{2}(X,\mathcal{O}_{X})\rightarrow 0
\end{equation*}
is exact and 
\begin{equation*}
\mathbb{D}(\widehat{\mathrm{Br}}_{X/R})_{W(R)}=\mathbb{D}(\hat{\mathbb{G}}_{m/R}^{h})_{W(R)}=H^{2}(\mathfrak{X},\mathcal{O}_{\mathfrak{X}})=\mathrm{Lie}\,\widehat{\mathrm{Br}}_{\mathfrak{X}/W(R)}=H^{2}(X,W\mathcal{O}_{X})\,.
\end{equation*}
Since $H^{i}(X,W\mathcal{O}_{X})=0$ for $i>2$, the lemma is proved in the case of surfaces as well. 
\end{proof}

In the following we will prove Proposition \ref{HW} for $i=0,j=1$ for surfaces and abelian schemes.

Using the quasi-isomorphism $N^{2}W\Omega_{X/R}^{\bullet}\cong\mathcal{F}^{2}\Omega_{\mathfrak{X}/W(R)}^{\bullet}$ and the fact that the $E_{1}$-hypercohomology spectral sequence associated to $\mathcal{F}^{2}\Omega_{\mathfrak{X}/W(R)}^{\bullet}$ degenerates, we compute the cohomology of the Nygaard complex (compare \cite[Remark 42]{LZ19}):
\\
\begin{itemize}
\item [--] $\mathbb{H}^{0}(N^{2}W\Omega_{X/R}^{\bullet})\cong I_{R}H^{0}(\mathfrak{X},\mathcal{O}_{\mathfrak{X}})=I_{R}W(R)$, \\
\item [--] $\mathbb{H}^{1}(N^{2}W\Omega_{X/R}^{\bullet})\cong I_{R}H^{1}(\mathfrak{X},\mathcal{O}_{\mathfrak{X}})\oplus I_{R}H^{0}(\mathfrak{X},\Omega_{\mathfrak{X}/W(R)}^{1})$, \\
\item [--] $\mathbb{H}^{2}(N^{2}W\Omega_{X/R}^{\bullet})\cong I_{R}H^{2}(\mathfrak{X},\mathcal{O}_{\mathfrak{X}})\oplus I_{R}H^{1}(\mathfrak{X},\Omega_{\mathfrak{X}/W(R)}^{1})\oplus H^{0}(\mathfrak{X},\Omega_{\mathfrak{X}/W(R)}^{2})$, \\
\item [--] $\mathbb{H}^{3}(N^{2}W\Omega_{X/R}^{\bullet})\cong I_{R}H^{3}(\mathfrak{X},\mathcal{O}_{\mathfrak{X}})\oplus I_{R}H^{2}(\mathfrak{X},\Omega_{\mathfrak{X}/W(R)}^{1})\oplus H^{1}(\mathfrak{X},\Omega_{\mathfrak{X}/W(R)}^{2})$ \\ \text{ \ \ \ \ \ } \hspace{21mm} $H^{0}(\mathfrak{X},\Omega_{\mathfrak{X}/W(R)}^{3})$, \\
\item [--] $\mathbb{H}^{4}(N^{2}W\Omega_{X/R}^{\bullet})\cong I_{R}H^{4}(\mathfrak{X},\mathcal{O}_{\mathfrak{X}})\oplus I_{R}H^{3}(\mathfrak{X},\Omega_{\mathfrak{X}/W(R)}^{1})\oplus H^{2}(\mathfrak{X},\Omega_{\mathfrak{X}/W(R)}^{2})$ \\ \text{ \ \ \ \ \ } \hspace{21mm} $\oplus H^{1}(\mathfrak{X},\Omega_{\mathfrak{X}/W(R)}^{3})\oplus H^{0}(\mathfrak{X},\Omega_{\mathfrak{X}/W(R)}^{4})$.
\end{itemize}

Since the map $\partial:H^{0}(X,W\mathcal{O}_{X})\cong W(R)\rightarrow\mathbb{H}^{0}(W\Omega_{X/R}^{1}\xrightarrow{dV}W\Omega_{X/R}^{2})$ is induced by the differential $d$ which vanishes on $W(R)$, $\partial$ is the zero map. Then we have the following commutative diagram
\begin{equation*}
\begin{tikzpicture}[descr/.style={fill=white,inner sep=1.5pt}]
        \matrix (m) [
            matrix of math nodes,
            row sep=4em,
            column sep=2em,
            text height=1.5ex, text depth=0.25ex
        ]
        { \mathbb{H}^{0}(W\Omega^{1}\xrightarrow{d}W\Omega^{2}) & H^{1}_{\rm{cris}}(X/W(R)) & H^{1}(X,W\mathcal{O}_{X}) \\
       \mathbb{H}^{0}(W\Omega^{1}\xrightarrow{dV}W\Omega^{2}) & \mathbb{H}^{1}(N^{2}W\Omega_{X/R}^{\bullet}) & H^{1}(X,W\mathcal{O}_{X})  \\};

        \path[overlay,->, font=\scriptsize] 
        (m-2-1) edge node [right]{$(V,\mathrm{id})$} (m-1-1)
        (m-2-3) edge node [right] {$pV$} (m-1-3)
        (m-2-2) edge (m-1-2)
        ;
        
        \path[overlay, ->>, font=scriptsize]
        (m-1-2) edge (m-1-3)
        (m-2-2) edge (m-2-3)
        ;
        
        \path[overlay, right hook->, font=scriptsize]
        (m-1-1) edge (m-1-2)
        (m-2-1) edge (m-2-2);
                        
\end{tikzpicture}
\end{equation*}
We will see below that $H^{1}(X,W\mathcal{O}_{X})=\mathbb{D}(\widehat{\mathrm{Pic}}_{X/R})$ is a direct summand of $H_{\rm{cris}}^{1}(X/W(R))$, hence the upper right arrow is a surjection. It is easy to see that the composite map
\begin{equation*}
I_{R}H^{1}(\mathfrak{X},\mathcal{O}_{\mathfrak{X}})\rightarrow\mathbb{H}^{1}(\mathcal{F}^{2}\Omega_{\mathfrak{X}/W(R)}^{\bullet})\cong\mathbb{H}^{1}(N^{2}W\Omega_{X/R}^{\bullet})\rightarrow H^{1}(X,W\mathcal{O}_{X})\xrightarrow{V}VH^{1}(X,W\mathcal{O}_{X})
\end{equation*}
can be identified with the isomorphism $I_{R}H^{1}(\mathfrak{X},\mathcal{O}_{\mathfrak{X}})\xrightarrow{\sim}VH^{1}(X,W\mathcal{O}_{X})$ constructed earlier. The lower right arrow in the diagram can then be identified with the surjective map $\mathbb{H}^{1}(\mathcal{F}^{2}\Omega_{\mathfrak{X}/W(R)}^{\bullet})\rightarrow I_{R}H^{1}(\mathfrak{X},\mathcal{O}_{\mathfrak{X}})$.

Since $V$ is injective on $H^{1}(X,W\mathcal{O}_{X})$ and $VH^{1}(X,W\mathcal{O}_{X})\cong I_{R}H^{1}(\mathfrak{X},\mathcal{O}_{\mathfrak{X}})$, the lower right arrow is surjective too. The left vertical arrow can be identified with the map (compare \eqref{comm square 2})
\begin{equation*}
\mathbb{H}^{0}(I_{R}\Omega_{\mathfrak{X}/W(R)}^{1}\xrightarrow{d}\Omega_{\mathfrak{X}/W(R)}^{2})=I_{R}H^{0}(\mathfrak{X},\Omega_{\mathfrak{X}}^{1})\rightarrow H^{0}(\mathfrak{X},\Omega_{\mathfrak{X}}^{1}) 
\end{equation*} 
which is injective and has cokernel $H^{0}(X,\Omega_{X/R}^{1})$. Then the commutative diagram
\\
\begin{equation}
\label{diagram}
\begin{adjustbox}{width=12cm}
\begin{tikzpicture}[descr/.style={fill=white,inner sep=1.5pt}]
        \matrix (m) [
            matrix of math nodes,
            row sep=4em,
            column sep=2em,
            text height=1.5ex, text depth=0.25ex
        ]
        { \mathbb{H}^{0}(W\Omega^{1}\xrightarrow{d}W\Omega^{2}) & H^{0}(X,W\Omega^{1}) & H^{0}(X,W\Omega^{2}) & \mathbb{H}^{1}(W\Omega^{1}\xrightarrow{d}W\Omega^{2}) \\
       \mathbb{H}^{0}(W\Omega^{1}\xrightarrow{dV}W\Omega^{2}) & H^{0}(X,W\Omega^{1}) & H^{0}(X,W\Omega^{2}) & \mathbb{H}^{1}(W\Omega^{1}\xrightarrow{dV}W\Omega^{2}) \\};

        \path[overlay,->, font=\scriptsize] 
        (m-2-1) edge node [right]{$(V,\mathrm{id})$} (m-1-1)
        (m-2-3) edge node [right] {$=$} (m-1-3)
        (m-2-2) edge node[right]{$V$} (m-1-2)
        (m-1-2) edge (m-1-3)
        (m-2-2) edge (m-2-3)
        (m-1-3) edge (m-1-4)
        (m-2-3) edge (m-2-4)
        (m-2-4) edge node [right]{$(V,\mathrm{id})$} (m-1-4)
        ;

        \path[overlay, right hook->, font=scriptsize]
        (m-1-1) edge (m-1-2)
        (m-2-1) edge (m-2-2);
                        
\end{tikzpicture}
\end{adjustbox}
\end{equation}
together with the injectivity of the map $(V,\mathrm{id})$ in \cite[Lemma 44]{LZ19} shows that $V$ is injective on $H^{0}(X,W\Omega_{X/R}^{1})$ and the cokernel is $H^{0}(X,\Omega_{X/R}^{1})$.

Let us now derive the Hodge-Witt decomposition of $H_{\rm{cris}}^{1}(X/W(R))$. Consider the canonical map
\begin{equation*}
H^{1}_{\mathrm{cris}}(\mathrm{Alb}_{X/R}/W(R))\rightarrow H^{1}_{\mathrm{cris}}(X/W(R))
\end{equation*}
induced by the Albanese morphism $X\rightarrow\mathrm{Alb}_{X/R}$. The first crystalline cohomology of an abelian scheme is the Dieudonn\'{e} crystal of the $p$-divisible group of the dual abelian scheme by \cite[Ch. 2]{MM74} (see also \cite[Th\'{e}orem\`{e} 2.5.6]{BBM82}), so we have a direct sum decomposition
\begin{equation}\label{direct sum bbm}
H_{\rm{cris}}^{1}(\mathrm{Alb}_{X/R}/W(R))\simeq\mathbb{D}(\mathrm{Pic}_{X/R}(p))=\mathbb{D}(\widehat{\mathrm{Pic}}_{X/R})\oplus\mathbb{D}(\mathrm{Pic}_{X/R}(p)^{\mathrm{\acute{e}t}})
\end{equation}
into a direct sum of Dieudonn\'{e} modules associated to the connected and \'{e}tale part of the $p$-divisible group associated to $\mathrm{Pic}_{X/R}=\left(\mathrm{Alb}_{X/R}\right)^{\vee}$. The induced map
\begin{equation*}
\mathbb{D}(\widehat{\mathrm{Pic}}_{X/R})\rightarrow H^{1}_{\mathrm{cris}}(X/W(R))\rightarrow H^{1}(X,W\mathcal{O}_{X})
\end{equation*}
is a map of Dieudonn\'{e} modules. The induced map 
\begin{equation*}
\mathbb{D}(\widehat{\mathrm{Pic}}_{X/R})/I_{R}\mathbb{D}(\widehat{\mathrm{Pic}}_{X/R})\rightarrow H^{1}(X,W\mathcal{O}_{X})/\mathrm{im}V\simeq H^{1}(X,\mathcal{O}_{X})
\end{equation*}
is a homomorphism of free $R$-modules of rank $=\dim\mathrm{Pic}_{X/R}$, and it is an isomorphism because it is an isomorphism after base change along $R\rightarrow k$ since $X_{k}$ is ordinary and using \cite[Remarque 3.11.2]{Ill79}. Hence $\mathbb{D}(\widehat{\mathrm{Pic}}_{X/R})\rightarrow H^{1}(X,W\mathcal{O}_{X})$ is an isomorphism by Lemma \ref{(iii)}. Since $\mathbb{D}(\widehat{\mathrm{Pic}}_{X/R})$ carries the structure of a multiplicative display (because $\widehat{\mathrm{Pic}}_{X/R}$ is multiplicative, since $X_{k}$ is ordinary), the isomorphism imposes the structure of a multiplicative display on $H^{1}(X,W\mathcal{O}_{X})$. It identifies $H^{1}(X,W\mathcal{O}_{X})$ as a direct summand of $H_{\mathrm{cris}}^{1}(X/W(R))$.

Now consider the induced map
\begin{equation*}
\eta:\mathcal{P}^{\mathrm{\acute{e}t}}=\mathbb{D}(\mathrm{Pic}_{X/R}(p)^{\mathrm{\acute{e}t}})\rightarrow\mathbb{H}^{0}(X,W\Omega_{X/R}^{\geq 1})\rightarrow H^{0}(X,W\Omega_{X/R}^{1})
\end{equation*}
which is compatible with the Frobenius on the left and $pF$ on $H^{0}(X,W\Omega_{X/R}^{1})$. We already have a commutative diagram
\begin{equation*}
\begin{tikzpicture}[descr/.style={fill=white,inner sep=1.5pt}]
        \matrix (m) [
            matrix of math nodes,
            row sep=4em,
            column sep=2em,
            text height=1.5ex, text depth=0.25ex
        ]
        { \mathcal{P}^{\mathrm{\acute{e}t}} & H^{0}(X,W\Omega_{X/R}^{1})  \\
       \mathcal{P}^{\mathrm{\acute{e}t}} & H^{0}(X,W\Omega_{X/R}^{1}) \\};

        \path[overlay,->, font=\scriptsize] 
        (m-1-1) edge node [above]{$\eta$} (m-1-2)
        (m-2-1) edge node [above]{$\eta$} (m-2-2)
        (m-1-1) edge node [left]{$F_{1}$} (m-2-1)
        (m-1-2) edge node [right]{$F$} (m-2-2)
        ;
                                
\end{tikzpicture}
\end{equation*}
Since $\mathcal{P}^{\mathrm{\acute{e}t}}$ is an \'{e}tale display, $F_{1}$ is defined on the the whole of $\mathcal{P}^{\mathrm{\acute{e}t}}$. On the versal deformation of $X$ the diagram commutes because it commutes after multiplication by $p$ and $p$ is injective on the versal deformation. Then the diagram
\begin{equation*}
\begin{tikzpicture}[descr/.style={fill=white,inner sep=1.5pt}]
        \matrix (m) [
            matrix of math nodes,
            row sep=4em,
            column sep=2em,
            text height=1.5ex, text depth=0.25ex
        ]
        { \mathcal{P}^{\mathrm{\acute{e}t}} & H^{0}(X,W\Omega_{X/R}^{1})  \\
       I_{R}\mathcal{P}^{\mathrm{\acute{e}t}} & H^{0}(X,W\Omega_{X/R}^{1}) \\};

        \path[overlay,->, font=\scriptsize] 
        (m-1-1) edge node [above]{$\eta$} (m-1-2)
        (m-2-1) edge node [above]{$\eta$} (m-2-2)
        (m-2-1) edge node [left]{$F_{2}$} (m-1-1)
        (m-1-2) edge node [right]{$V$} (m-2-2)
        ;
                                
\end{tikzpicture}
\end{equation*}
with $F_{2}(\tensor[^V]{\xi}{}x)=\xi F_{1}x$ commutes, hence we get an induced map of free $R$-modules
\begin{equation*}
\overline{\eta}:\mathcal{P}^{\mathrm{\acute{e}t}}/I_{R}\mathcal{P}^{\mathrm{\acute{e}t}}\rightarrow H^{0}(X,W\Omega_{X/R}^{1})/\mathrm{im}\,V\cong H^{0}(X,\Omega_{X/R}^{1})
\end{equation*}
of rank $=\dim\mathrm{Pic}^{0}_{X/R}$. It is enough to show that $\overline{\eta}$ is surjective to show that it is an isomorphism. We show this after base change along $R\rightarrow k$. Over $k$ it is known that $\mathbb{D}(\mathrm{Pic}_{X_{k}}(p)^{\mathrm{\acute{e}t}})$ is the slope $p$-part of in $H_{\rm{cris}}^{1}(X_{k}/W(k))$, hence is isomorphic to $H^{0}(X_{k},W\Omega_{X_{k}/k}^{1})$ because we are in the ordinary case. Since the map $H^{0}(X_{k},W\Omega_{X_{k}/k}^{1})\rightarrow H^{0}(X_{k},\Omega_{X_{k}/k}^{1})$ is surjective, $\overline{\eta}$ is surjective and hence an isomorphism. Applying Lemma \ref{(iii)} shows that $\eta:\mathbb{D}(\mathrm{Pic}_{X/R}(p)^{\mathrm{\acute{e}t}})\rightarrow H^{0}(X,W\Omega_{X/R}^{1})$ is an isomorphism. One consequence of the case $i=0$, $j=1$ is that $\mathbb{H}^{0}(W\Omega^{1}_{X/R}\xrightarrow{d}W\Omega_{X/R}^{2})=H^{0}(X,W\Omega_{X/R}^{1})$ in diagram \eqref{diagram}. Indeed, the isomorphism $\eta$ induces a surjection $\mathbb{H}^{0}(X,W\Omega_{X/R}^{\geq 1})\rightarrow H^{0}(X,W\Omega_{X/R}^{1})$ so the map $\mathbb{H}^{0}(W\Omega_{X/R}^{1}\xrightarrow{d}W\Omega_{X/R}^{2})\rightarrow H^{0}(X,W\Omega_{X/R}^{1})$ is surjective too, and hence the identity. The commutativity of the diagram
\begin{equation*}
\begin{tikzpicture}[descr/.style={fill=white,inner sep=1.5pt}]
        \matrix (m) [
            matrix of math nodes,
            row sep=4em,
            column sep=2em,
            text height=1.5ex, text depth=0.25ex
        ]
        { H_{\mathrm{cris}}^{1}(X/W(R)) & H^{1}(X,W\mathcal{O}_{X})  \\
       H_{\mathrm{dR}}^{1}(\mathfrak{X}/W(R)) & H^{1}(\mathfrak{X},\mathcal{O}_{\mathfrak{X}}) \\};

        \path[overlay,->, font=\scriptsize] 
        (m-1-1) edge node [left]{$\cong$} (m-2-1)
        (m-1-2) edge node [right]{$\cong$} (m-2-2)
        ;
        
        \path[overlay,->>, font=\scriptsize]
        (m-1-1) edge (m-1-2)
        (m-2-1) edge (m-2-2)
        ;
                                
\end{tikzpicture}
\end{equation*}
implies the isomorphism $H^{0}(X,W\Omega^{1}_{X/R})\cong H^{0}(\mathfrak{X},\Omega_{\mathfrak{X}/W(R)}^{1})$. (Note that $H^{1}(\mathfrak{X},\mathcal{O}_{\mathfrak{X}})$ is the tangent space of $\widehat{\mathrm{Pic}}_{\mathfrak{X}/W(R)}$, hence is isomorphic to the value of the Dieudonn\'{e} crystal of $\widehat{Pic}_{\mathfrak{X}/W(R)}$ at $W(R)$, which is $H^{1}(X,W\mathcal{O}_{X})$ by rigidity of $\widehat{Pic}$). The decomposition \eqref{direct sum bbm} then reflects the Hodge-Witt decomposition in degree one. We have isomorphisms
\begin{equation*}
H^{1}_{\mathrm{cris}}(\mathrm{Alb}_{X/R}/W(R))\simeq H^{1}_{\mathrm{cris}}(X/W(R)\simeq\mathbb{D}(\widehat{\mathrm{Pic}}_{X/R})\oplus\mathbb{D}(\mathrm{Pic}_{X/R}(p)^{\mathrm{\acute{e}t}})\,.
\end{equation*}
The display structure on $H_{\rm{cris}}^{1}(X/W(R))$ arising from the Nygaard complex $N^{1}W\Omega_{X/R}^{\bullet}$ has been analysed in \cite[3.4]{LZ04}. We have $P=H_{\rm{cris}}^{1}(X/W(R))$, $Q=\ker(H_{\rm{cris}}^{1}(X/W(R))\rightarrow H^{1}(X,\mathcal{O}_{X})$, $F$ is the crystalline Frobenius, and $F_{1}$ is defined on the Nygaard filtration $Q\cong\mathbb{H}^{1}(X,N^{1}W\Omega_{X/R}^{\bullet})$. We see then that
\begin{equation*}
Q=I_{R}\mathbb{D}(\widehat{\mathrm{Pic}}_{X/R})\oplus\mathbb{D}(\mathrm{Pic}_{X/R}(p)^{\mathrm{\acute{e}t}})
\end{equation*}
because $H^{1}(X,\mathcal{O}_{X})$ is the tangent space of $\widehat{\mathrm{Pic}}_{X/R}$ and we have an exact sequence 
\begin{equation*}
0\rightarrow I_{R}\mathbb{D}(\widehat{\mathrm{Pic}}_{X/R})\rightarrow\mathbb{D}(\widehat{\mathrm{Pic}}_{X/R})\rightarrow H^{1}(X,\mathcal{O}_{X})\rightarrow 0\,.
\end{equation*}
On $I_{R}\mathbb{D}(\widehat{\mathrm{Pic}}_{X/R})$ the map $F_{1}$ is defined by $F_{1}(\tensor[^V]{\xi}{}\alpha)=\xi F\alpha$. Since $\mathbb{D}(\mathrm{Pic}_{X/R}(p)^{\mathrm{\acute{e}t}})$ is an \'{e}tale display, $F_{1}$ is defined on $\mathbb{D}(\mathrm{Pic}_{X/R}(p)^{\mathrm{\acute{e}t}})$ given that for a display of an \'{e}tale group one has ``$P=Q$''. We conclude that the above decomposition is a direct sum of displays. 

Before we can finish the proof of Theorem \ref{surfaces} we must prove Proposition \ref{HW} in the cases $i=1,2$, $j=1$, $\dim X=2$. For $i=1$ the proof is very similar as for \cite[Lemma 46]{LZ19}.

Then diagram (92) in \cite{LZ19} holds verbatim for general smooth projective surfaces:
\\
\begin{equation}
\label{diagram from LZ19}
\begin{adjustbox}{width=12cm}
\begin{tikzpicture}[descr/.style={fill=white,inner sep=1.5pt}]
        \matrix (m) [
            matrix of math nodes,
            row sep=4em,
            column sep=2em,
            text height=1.5ex, text depth=0.25ex
        ]
        { H^{0}(X,W\Omega^{2}) & \mathbb{H}^{1}(X,W\Omega^{1}\xrightarrow{d} W\Omega^{2}) & H^{1}(X,W\Omega^{1}) & H^{1}(X,W\Omega^{2}) \\
       H^{0}(X,W\Omega^{2}) & \mathbb{H}
        ^{1}(X,W\Omega^{1}\xrightarrow{dV}W\Omega^{2}) & H^{1}(X,W\Omega^{1}) & H^{1}(X,W\Omega^{2}) \\
       };

        \path[overlay,->, font=\scriptsize] 
        (m-2-1) edge node [right]{$=$} (m-1-1)
        (m-2-3) edge node [right] {$V$} (m-1-3)
        (m-2-2) edge node[right]{$\hat{\alpha}$} (m-1-2)
        (m-1-2) edge (m-1-3)
        (m-2-2) edge (m-2-3)
        (m-1-3) edge (m-1-4)
        (m-2-3) edge (m-2-4)
        (m-2-4) edge node [right]{$=$} (m-1-4)
        ;

        \path[overlay, right hook->, font=scriptsize]
        (m-1-1) edge (m-1-2)
        (m-2-1) edge (m-2-2);
                        
\end{tikzpicture}
\end{adjustbox}
\end{equation}
As in cohomological degree $1$ one sees that the composite map
\begin{equation*}
I_{R}H^{2}(\mathfrak{X},\mathcal{O}_{\mathfrak{X}})\rightarrow\mathbb{H}^{2}(\mathcal{F}^{2}\Omega_{\mathfrak{X}/W(R)}^{\bullet})\cong\mathbb{H}^{2}(N^{2}W\Omega_{X/R}^{\bullet})\rightarrow H^{2}(X,W\mathcal{O}_{X})\xrightarrow{V}VH^{2}(X,W\mathcal{O}_{X})
\end{equation*}
agrees with the isomorphism constructed earlier. Hence $\mathbb{H}^{2}(N^{2}W\Omega_{X/R}^{\bullet})\rightarrow H^{2}(X,W\mathcal{O}_{X})$ is surjective. 
Then we have a commutative diagram of isomorphisms
\begin{equation*}
\begin{tikzpicture}[descr/.style={fill=white,inner sep=1.5pt}]
        \matrix (m) [
            matrix of math nodes,
            row sep=4em,
            column sep=2em,
            text height=1.5ex, text depth=0.25ex
        ]
        { 0 & \mathbb{H}^{2}(X,W\Omega_{X/R}^{1}\xrightarrow{d}W\Omega_{X/R}^{2}) & H_{\rm{cris}}^{3}(X/W(R))  \\
      0 & \mathbb{H}^{2}(X,W\Omega_{X/R}^{1}\xrightarrow{dV}W\Omega_{X/R}^{2}) & \mathbb{H}^{3}(X,N^{2}W\Omega_{X/R}^{\bullet})  \\};

        \path[overlay,->, font=\scriptsize] 
        (m-1-2) edge node[above]{$\cong$}(m-1-3)
        (m-2-2) edge node[above]{$\cong$}(m-2-3)
        (m-2-2) edge node [left]{$(V,\mathrm{id})$} (m-1-2)
        (m-2-3) edge (m-1-3)
        (m-1-1) edge (m-1-2)
        (m-2-1) edge (m-2-2)
        ;
                                
\end{tikzpicture}
\end{equation*}
and the right vertical map can be identified under the isomorphism between crystalline and de Rham cohomology with the injection
\begin{equation*}
I_{R}H^{2}(\mathfrak{X},\Omega_{\mathfrak{X}/W(R)}^{1})\oplus H^{1}(\mathfrak{X},\Omega_{\mathfrak{X}/W(R)}^{2})\rightarrow H^{2}(\mathfrak{X},\Omega_{\mathfrak{X}/W(R)}^{1})\oplus H^{1}(\mathfrak{X},\Omega_{\mathfrak{X}/W(R)}^{2})\,,
\end{equation*}
hence $(V,\mathrm{id})$ is injective in the above diagram. This fact together with the injectivity of $\hat{\alpha}$ in diagram \eqref{diagram from LZ19} imply that $V$ is injective on $H^{1}(X,W\Omega_{X/R}^{1})$ and $\mathrm{coker} V\cong H^{1}(X,\Omega_{X/R}^{1})$. To finish the proof of Proposition \ref{HW} in the case $i=1$, $j=1$, we can apply \cite[Lemma 47]{LZ19}. Note that we also have $(P,Q,F,F_{1})=\mathbb{D}(\Phi_{X/R}^{\acute{e}t})$, the display of the \'{e}tale part of the extended formal Brauer group $\Phi_{X/R}$, and $\mathbb{D}(\Phi_{X/R}^{\acute{e}t})\xrightarrow{\simeq} H^{1}(X,W\Omega_{X/R}^{1})$ where the map is given by the composite map
\begin{equation*}
\mathbb{D}(\Phi_{X/R}^{\acute{e}t})\rightarrow\mathbb{H}^{2}(X,W\Omega_{X/R}^{\geq 1})\rightarrow H^{1}(X,W\Omega_{X/R}^{1})\,.
\end{equation*}
\cite[Lemma 47]{LZ19} implies that $\mathbb{H}^{2}(X,W\Omega_{X/R}^{\geq 1})\rightarrow H^{1}(X,W\Omega_{X/R}^{1})$ is surjective, hence we get a commutative diagram
\begin{equation*}
\begin{adjustbox}{width=12.5cm}
\begin{tikzpicture}[descr/.style={fill=white,inner sep=1.5pt}]
        \matrix (m) [
            matrix of math nodes,
            row sep=4em,
            column sep=1.5em,
            text height=1.5ex, text depth=0.25ex
        ]
        { H^{1}(X,W\Omega^{2}) & \mathbb{H}^{2}(X,W\Omega^{1}\xrightarrow{d} W\Omega^{2}) & H^{2}(X,W\Omega^{1}) & H^{2}(X,W\Omega^{2})  & \mathbb{H}^{3}(X,W\Omega^{1}\xrightarrow{d} W\Omega^{2}) \\
       H^{1}(X,W\Omega^{2}) & \mathbb{H}^{2}(X,W\Omega^{1}\xrightarrow{dV} W\Omega^{2}) & H^{2}(X,W\Omega^{1}) & H^{2}(X,W\Omega^{2})  & \mathbb{H}^{3}(X,W\Omega^{1}\xrightarrow{dV} W\Omega^{2}) \\
       };

        \path[overlay,->, font=\scriptsize] 
        (m-2-1) edge node [right]{$=$} (m-1-1)
        (m-2-3) edge node [right] {$V$} (m-1-3)
        (m-2-2) edge node[right]{$(V,\mathrm{id})$} (m-1-2)
        (m-1-2) edge (m-1-3)
        (m-2-2) edge (m-2-3)
        (m-1-3) edge (m-1-4)
        (m-2-3) edge (m-2-4)
        (m-2-4) edge node [right]{$=$} (m-1-4)
        (m-2-5) edge node [right]{$(V,\mathrm{id})$} (m-1-5)
        (m-1-4) edge (m-1-5)
        (m-2-4) edge (m-2-5)
        ;

        \path[overlay, right hook->, font=scriptsize]
        (m-1-1) edge (m-1-2)
        (m-2-1) edge (m-2-2);
                        
\end{tikzpicture}
\end{adjustbox}
\end{equation*}
We have already seen that $(V,\mathrm{id})$ is injective in cohomological degree $2$; a similar argument shows that $(V,\mathrm{id})$ is injective in cohomological degree $3$ as well. Then the above diagram implies that $V$ is injective on $H^{2}(X,W\Omega_{X/R}^{1})$ and its cokernel is $H^{2}(X,\Omega_{X/R}^{1})$.

Let $\mathcal{A}lb(p)^{0}(1)$ be the (twisted) connected component of the $p$-divisible group associated to the Albanese scheme. It is known that this is the Cartier dual of the \'{e}tale $p$-divisible group $\mathrm{Pic}_{X/R}(p)^{\mathrm{\acute{e}t}}$. Let $\mathbb{D}(\mathcal{A}lb(p)^{0})(1)$ be the associated display. Under the Poincar\'{e} duality pairing
\begin{equation*}
H_{\rm{cris}}^{1}(X/W(R))\times H_{\rm{cris}}^{3}(X/W(R))(2)\rightarrow W(R)
\end{equation*}
the dual of the map 
\begin{equation*}
\mathbb{D}(\mathrm{Pic}_{X/R}(p)^{\mathrm{\acute{e}t}})\rightarrow H_{\rm{cris}}^{1}(X/W(R))
\end{equation*}
is the map
\begin{equation*}
H_{\rm{cris}}^{3}(X/W(R))(2)\rightarrow\mathbb{D}(\mathcal{A}lb(p)^{0})(1)
\end{equation*}
hence $\mathbb{D}(\mathcal{A}lb(p)^{0})(-1)$ is a direct summand of 
\begin{equation*}
H_{\rm{cris}}^{3}(X/W(R))\cong\mathbb{H}^{2}(X,W\Omega_{X/R}^{1}\xrightarrow{d}W\Omega_{X/R}^{2})\,.
\end{equation*}
It induces a map
\begin{equation*}
\mathbb{D}(\mathcal{A}lb(p)^{0})(-1)\rightarrow H^{2}(X,W\Omega_{X/R}^{1})
\end{equation*}
where the Frobenius on the right is induced by the crystalline Frobenius, hence by $pF$, where $F$ is the Frobenius map on $W\Omega_{X/R}^{1}$. By an analogous argument as for $H^{i}(X,W\Omega_{X/R}^{1})$ ($i=0,1$) we get a homomorphism
\begin{equation*}
\varsigma\,:\,\mathbb{D}(\mathcal{A}lb(p)^{0})\rightarrow H^{2}(X,W\Omega_{X/R}^{1})
\end{equation*}
where $\mathbb{D}(\mathcal{A}lb(p)^{0})$ is the multiplicative display associated to $\mathcal{A}lb(p)^{0}$ and the Frobenius on $H^{2}(X,W\Omega_{X/R}^{1})$ is the one induced by $F$ on $W\Omega_{X/R}^{1}$. The induced map
\begin{equation*}
\mathbb{D}(\mathcal{A}lb(p)^{0})/I_{R}\mathbb{D}(\mathcal{A}lb(p)^{0})\rightarrow H^{2}(X,W\Omega_{X/R}^{1})/\mathrm{im}\,V\cong H^{2}(X,\Omega_{X/R}^{1})
\end{equation*}
is an isomorphism because it is so after base change along $R\rightarrow k$, since we are in the ordinary case. Hence $\varsigma$ is an isomorphism by Lemma \ref{(iii)}.

Now we can complete the proof of Theorem \ref{surfaces}:

Under the duality of $H_{\rm{cris}}^{1}$ and $H_{\rm{cris}}^{3}$ we get a direct summand decomposition
\begin{align}
H_{\rm{cris}}^{3}(X/W(R))
& =H^{2}(X,W\Omega_{X/R}^{1})\oplus\mathbb{D}(\mathcal{A}lb(p)^{\mathrm{\acute{e}t}})(-1)\\
& =\mathbb{D}(\mathcal{A}lb(p)^{0})(-1)\oplus\mathbb{D}(\mathcal{A}lb(p)^{\mathrm{\acute{e}t}}(-1)) 
\end{align}
of Dieudonn\'{e} modules, where $\mathcal{A}lb(p)^{\mathrm{\acute{e}t}}(1)$ is the Cartier dual of $\mathrm{Pic}_{X/R}^{0}(p)^{0}$. Since the cup product of $H^{1}(X,W\mathcal{O}_{X})$ with $H^{2}(X,W\Omega_{X/R}^{1})$ vanishes, we get an induced map
\begin{equation*}
\mathbb{D}(\mathcal{A}lb(p)^{\mathrm{\acute{e}t}})(-1)\rightarrow H^{1}(X,W\Omega_{X/R}^{2})
\end{equation*}
where the Frobenius on the right is induced by $p^{2}F$. It is clear that this map is an isomorphism too.

For $H_{\rm{cris}}^{2}$ we can follow the argument in the proof of \cite[Theorem 40]{LZ19} to get the Hodge-Witt decomposition 
\begin{equation*}
H_{\rm{cris}}^{2}(X/W(R))=H^{2}(X,W\mathcal{O}_{X})\oplus H^{1}(X,W\Omega_{X/R}^{1})\oplus H^{0}(X,W\Omega_{X/R}^{2})
\end{equation*} 
into a direct sum of displays associated to $\widehat{\mathrm{Br}}_{X/R}$, $\Phi_{X/R}^{\acute{e}t}$ and the twisted dual $\widehat{\mathrm{Br}}_{X/R}^{\ast}(-1)$. This finishes the proof of Theorem \ref{surfaces}.

\section{Hodge-Witt decomposition for abelian schemes}

For abelian schemes we reformulate Proposition \ref{HW} as follows:

\begin{prop}
Let $A$ be an abelian scheme over $\mathrm{Spec}\,R$, with $d=\dim A<p$, such that the closed fibre is ordinary. Fix a pair $(i,j)$ with $0\leq i,j\leq d$. Then we have
\begin{enumerate}[(i)]
\item There is an exact sequence
\begin{equation*}
0\rightarrow H^{i}(A,W\Omega_{A/R}^{j})\xrightarrow{V}H^{i}(A,W\Omega_{A/R}^{j})\rightarrow H^{i}(A,\Omega_{A/R}^{j})\rightarrow 0
\end{equation*}
induced by the action of $V$ on $W\Omega_{A/R}^{j}$.

\item We have canonical isomorphisms
\begin{equation*}
\bigwedge^{j}(H^{0}(A,W\Omega_{A/R}^{1}))\otimes\bigwedge^{i}(H^{1}(A,W\mathcal{O}_{A}))\xrightarrow{\varsigma}H^{i}(A,W\Omega_{A/R}^{j})
\end{equation*}
compatible with the Frobenius action $\bigwedge^{j}F\otimes\bigwedge^{i}F$ on the left and the Frobenius induced by $F$ on $W\Omega_{A/R}^{j}$ on the right, such that there are isomorphisms induced by the maps $\alpha_{r}$ in \eqref{alpha r}
\begin{equation*}
H^{i}(A,W\Omega_{A/R}^{j})\cong H^{i}(\mathcal{A},\Omega_{\mathcal{A}/W(R)}^{j})
\end{equation*} 
with isomorphisms 
\begin{equation*}
VH^{i}(A,W\Omega_{A/R}^{j})\cong I_{R}H^{i}(\mathcal{A},\Omega_{\mathcal{A}/W(R)}^{j})\,.
\end{equation*} 
In particular,
\begin{equation*}
(P,Q,F,F_{1})=(H^{i}(A,W\Omega_{A/R}^{j}),VH^{i}(A,W\Omega_{A/R}^{j}),F,V^{-1})
\end{equation*}
is a multiplicative display.
\end{enumerate} 
\end{prop}
\begin{proof}
We prove this by induction on $j$, the case $j=0$ having already been covered. Assume that the proposition holds for all $j<r$ and all $i$. Consider the exact sequence
\begin{equation*}
\begin{tikzpicture}[descr/.style={fill=white,inner sep=1.5pt}]
        \matrix (m) [
            matrix of math nodes,
            row sep=4em,
            column sep=1.5em,
            text height=1.5ex, text depth=0.25ex
        ]
        { 0 & \mathbb{H}^{i-1}(A,N^{r+1}W\Omega_{A/R}^{\geq 1}) & \mathbb{H}^{i}(A,N^{r+1}W\Omega_{A/R}^{\bullet}) & H^{i}(A,W\mathcal{O}_{A}) & 0 \\
       \ & \ & \mathbb{H}^{i}(\mathcal{A},\mathcal{F}^{r+1}\Omega_{\mathcal{A}/W(R)}^{\bullet}) & I_{R}H^{i}(\mathcal{A},\mathcal{O}_{\mathcal{A}}) & \ \\
       };

        \path[overlay,->, font=\scriptsize] 
        (m-1-1) edge (m-1-2)
        (m-2-3) edge node [right] {$\cong$} (m-1-3)
        (m-1-2) edge (m-1-3)
        (m-1-3) edge (m-1-4)
        (m-2-3) edge (m-2-4)
        (m-2-4) edge node[left]{$\cong$} node [right]{$\tensor[^F]{\alpha}{_0}$} (m-1-4)
        (m-1-4) edge (m-1-5)
        ;     
                                
\end{tikzpicture}
\end{equation*}
where the last map is surjective because the $E_{1}$-hypercohomology spectral sequence associated to $\mathcal{F}^{r+1}\Omega_{\mathcal{A}/W(R)}^{\bullet}$ degenerates. The same argument applies to the previous cohomological degree $i-1$, hence the first map is injective, and we have
\begin{equation*}
\mathbb{H}^{i-1}(A,N^{r+1}W\Omega_{A/R}^{\geq 1})\cong \mathbb{H}^{i-1}(\mathcal{A},\mathcal{F}^{r+1}\Omega_{\mathcal{A}/W(R)}^{\geq 1})\,.
\end{equation*}
By induction one proves that for all $s\leq r$ we have
\begin{equation*}
\mathbb{H}^{i-s}(A,N^{r+1}W\Omega_{A/R}^{\geq s})\cong\mathbb{H}^{i-s}(\mathcal{A},\mathcal{F}^{r+1}\Omega_{\mathcal{A}/W(R)}^{\geq s})\,.
\end{equation*}
Indeed, we have an exact sequence
\begin{equation*}
\begin{adjustbox}{width=12.5cm}
\begin{tikzpicture}[descr/.style={fill=white,inner sep=1.5pt}]
        \matrix (m) [
            matrix of math nodes,
            row sep=4em,
            column sep=1.5em,
            text height=1.5ex, text depth=0.25ex
        ]
        { 0 & \mathbb{H}^{i-(s+1)}(A,N^{r+1}W\Omega_{A/R}^{\geq s+1}) & \mathbb{H}^{i-s}(A,N^{r+1}W\Omega_{A/R}^{\geq s}) & H^{i-s}(A,W\Omega_{A/R}^{s}) & 0 \\
       \ & \ & \mathbb{H}^{i-s}(\mathcal{A},\mathcal{F}^{r+1}\Omega_{\mathcal{A}/W(R)}^{\geq s}) & I_{R}H^{i-s}(\mathcal{A},\Omega_{\mathcal{A}/W(R)}^{s}) & \ \\
       };

        \path[overlay,->, font=\scriptsize] 
        (m-1-1) edge (m-1-2)
        (m-2-3) edge node [right] {$\cong$} (m-1-3)
        (m-1-2) edge (m-1-3)
        (m-1-3) edge (m-1-4)
        (m-2-3) edge (m-2-4)
        (m-2-4) edge node[left]{$\cong$} node [right]{$\tensor[^F]{\alpha}{_s}$} (m-1-4)
        (m-1-4) edge (m-1-5)
        ;     
                                
\end{tikzpicture}
\end{adjustbox}
\end{equation*}
and hence we get
\begin{equation*}
\mathbb{H}^{i-(s+1)}(A,N^{r+1}W\Omega_{A/R}^{\geq s+1})\cong\mathbb{H}^{i-(s+1)}(\mathcal{A},\mathcal{F}^{r+1}\Omega_{\mathcal{A}/W(R)}^{\geq s+1})\,.
\end{equation*}
We conclude under this isomorphism the map 
\begin{equation*}
\mathbb{H}^{i}(A,N^{r+1}W\Omega_{A/R}^{\geq r})\xrightarrow{(V,\mathrm{id})}\mathbb{H}^{i}(A,W\Omega_{A/R}^{\geq r})
\end{equation*}
corresponds to 
\begin{equation*}
\mathbb{H}^{i}(\mathcal{A},\mathcal{F}^{r+1}\Omega_{\mathcal{A}/W(R)}^{\geq r})\rightarrow\mathbb{H}^{i}(\mathcal{A},\Omega_{\mathcal{A}/W(R)}^{\geq r})
\end{equation*}
and hence is injective and the cokernel is isomorphic to $H^{i}(A,\Omega_{A/R}^{r})$ (see \eqref{comm square 2}).

Now consider the diagram
\begin{equation*}
\begin{tikzpicture}[descr/.style={fill=white,inner sep=1.5pt}]
        \matrix (m) [
            matrix of math nodes,
            row sep=4em,
            column sep=2em,
            text height=1.5ex, text depth=0.25ex
        ]
        { \mathbb{H}^{i}(W\Omega^{\geq r+1}) & \mathbb{H}^{i}(W\Omega^{\geq r}) & H^{i}(W\Omega^{r}) & \mathbb{H}^{i+1}(W\Omega^{\geq r+1}) & \ \\
      \mathbb{H}^{i}(W\Omega^{\geq r+1}) & \mathbb{H}^{i}(N^{r+1}W\Omega^{\geq r}) & H^{i}(W\Omega^{r}) & \mathbb{H}^{i+1}(W\Omega^{\geq r+1}) & \  \\
       };

        \path[overlay,->, font=\scriptsize] 
        (m-2-1) edge node [left]{$=$} node [right]{$\mathrm{id}$} (m-1-1)
        (m-2-3) edge node [right] {$V$} (m-1-3)
        (m-2-2) edge node[right]{$(V,\mathrm{id})$} (m-1-2)
        (m-1-2) edge (m-1-3)
        (m-2-2) edge (m-2-3)
        (m-1-3) edge (m-1-4)
        (m-2-3) edge (m-2-4)
        (m-2-4) edge node [right]{$\mathrm{id}$} node [left]{$=$} (m-1-4)
        (m-1-4) edge (m-1-5)
        (m-2-4) edge (m-2-5)
        (m-1-1) edge (m-1-2)
        (m-2-1) edge (m-2-2);
                        
\end{tikzpicture}
\end{equation*}
The injectivity of $(V,\mathrm{id})$ in cohomological degrees $i$ and $i+1$ implies that $V$ is injective on $H^{i}(A,W\Omega_{A/R}^{r})$ and has cokernel $H^{i}(A,\Omega_{A/R}^{r})$ as desired.

For the second part of the proposition, we have a commutative diagram (compare the cases $j=0$ and $i=0,j=1$)
\begin{equation*}
\begin{tikzpicture}[descr/.style={fill=white,inner sep=1.5pt}]
        \matrix (m) [
            matrix of math nodes,
            row sep=4em,
            column sep=2em,
            text height=1.5ex, text depth=0.25ex
        ]
        { \displaystyle\bigwedge\limits^{j}H^{0}(A,W\Omega_{A/R}^{1})\otimes\displaystyle\bigwedge\limits^{i}H^{1}(A,W\mathcal{O}_{A}) & H^{i}(A,W\Omega_{A/R}^{j}) & H^{i}(A,\Omega_{A/R}^{j}) \\
       \displaystyle\bigwedge\limits^{j}H^{0}(\mathcal{A},\Omega_{\mathcal{A}/W(R)}^{1})\otimes\displaystyle\bigwedge\limits^{i}H^{1}(\mathcal{A},\mathcal{O}_{\mathcal{A}}) & H^{i}(\mathcal{A},\Omega_{\mathcal{A}}^{j}) & \ \\};

        \path[overlay,->, font=\scriptsize] 
        (m-1-1) edge node [above]{$\varsigma$} (m-1-2)
        (m-2-1) edge node [above]{$\simeq$} (m-2-2)
        (m-2-1) edge node [left]{$\simeq$} (m-1-1)
        (m-1-2) edge (m-1-3)
        ;
        
        \path[overlay,->>, font=\scriptsize]
        (m-2-2) edge (m-1-3)
        ;
                        
\end{tikzpicture}
\end{equation*}
where the horizontal maps are cup products in cohomology.

We define $(P,Q=I_{R}P,F,F_{1})$ to be the multiplicative display given by setting $P=\displaystyle\bigwedge\limits^{j}H^{0}(A,W\Omega_{A/R}^{1})\otimes\bigwedge^{i}H^{1}(A,W\mathcal{O}_{A})$. Since $V\circ\varsigma\circ F_{1}=\varsigma|_{I_{R}P}$ we get the induced homomorphism of free $R$-modules of rank $h^{j,i}$
\begin{equation*}
\overline{\varsigma}\,:\, P/I_{R}P\rightarrow H^{i}(A,W\Omega_{A/R}^{j})/\mathrm{im}\,V\cong H^{i}(A,\Omega_{A/R}^{j})
\end{equation*}
which coincides with the cup product map
\begin{equation*}
\displaystyle\bigwedge\limits^{j}H^{0}(A,\Omega_{A/R}^{1})\otimes\bigwedge^{i}H^{1}(A,\mathcal{O}_{A})\rightarrow H^{i}(A,\Omega_{A/R}^{j})\,.
\end{equation*}
Since for abelian varieties we have $H_{\mathrm{dR}}^{s}(A/R)=\bigwedge^{s}H_{\mathrm{dR}}^{1}(A/R)$ the map $\overline{\varsigma}$ is an isomorphism. Therefore $\varsigma$ is an isomorphism by Lemma \ref{(iii)}.

Since the constructions of the maps $\alpha_{r}$ and $\tensor[^F]{\alpha}{_r}$ in section 3 are compatible with taking tensor products of complexes, the isomorphism $H^{i}(\mathcal{A},\Omega_{\mathcal{A}/W(R)}^{j})\xrightarrow{\sim}H^{i}(A,W\Omega_{A/R}^{j})$ obtained above coincides with the canonical map $\alpha_{j}$. 
\end{proof}

This finishes the proof of the Hodge-Witt decomposition of $H_{\rm{cris}}^{s}(A/W(R))$, i.e. Theorem \ref{abelian} holds. Note that the isomorphism
\begin{equation*}
\alpha_{j}:H^{i}(\mathcal{A},\Omega_{\mathcal{A}/W(R)}^{j})\rightarrow\mathbb{H}^{i}(\mathcal{A},\Omega_{\mathcal{A}/W(R)}^{\geq j})\rightarrow\mathbb{H}^{i}(A,W\Omega^{\geq j}_{A/R})\rightarrow H^{i}(A,W\Omega_{A/R}^{j})
\end{equation*}
splits the surjection $\mathbb{H}^{i}(A,W\Omega^{\geq j}_{A/R})\twoheadrightarrow H^{i}(A,W\Omega_{A/R}^{j})$ for all $j$ and hence induces the Hodge-Witt decomposition on crystalline cohomology.

\section{Hodge-Witt decomposition for $n$-folds}\label{n-fold section}

Let $\mathcal{S}$ be a smooth formal scheme over $\mathrm{Spf}\,W(k)$, and let $f\,:\,\mathcal{X}\rightarrow\mathcal{S}$ be a smooth and proper family of $n$-folds, where $n<p$. Suppose that the conditions (B1) and (B2) from the introduction are satisfied. In this section we shall prove the following theorem (Theorem \ref{general} from the introduction):

\begin{thm}\label{deformation}
Let $X_{k}:=\mathcal{X}\times_{\mathcal{S}}\mathrm{Spec}\,k$ be the fibre over an ordinary $k$-point $\mathrm{Spec}\,k\rightarrow\mathcal{S}$. Then for any commutative diagram
\begin{equation*}
\begin{tikzpicture}[descr/.style={fill=white,inner sep=1.5pt}]
        \matrix (m) [
            matrix of math nodes,
            row sep=2.5em,
            column sep=2.5em,
            text height=1.5ex, text depth=0.25ex
        ]
        { \mathrm{Spec}\,R & \mathcal{S}  \\
       \mathrm{Spec}\,k  & \ \\};

        \path[overlay,->, font=\scriptsize] 
        (m-1-1) edge (m-1-2)
        (m-2-1) edge (m-1-2)
        ;
        
        \path[overlay, right hook->, font=\scriptsize]
        (m-2-1) edge (m-1-1)
        ;
                                        
\end{tikzpicture}
\end{equation*}
where $R$ is an artinian local ring with residue field $k$, the deformation $X:=\mathcal{X}\times_{\mathcal{S}}\mathrm{Spec}\,R$ of $X_{k}$ admits a Hodge-Witt decomposition of $H_{\mathrm{cris}}^{s}(X/W(R))$ as displays in all degrees $0\leq s\leq 2n$.
\end{thm}

\begin{rem}\label{assumptions}
The standard assumptions (B1) and (B2) ensure that the relative Hodge filtration of the family is a filtration by locally direct factors, and commutes with arbitrary base change $\mathcal{S}\rightarrow\mathcal{T}$ \cite[Corollary 8.3]{Kat70}. Notice that (B1) and (B2) ensures that the deformation $X/\mathrm{Spec}\,R$ satisfies (A1) and (A2). Indeed, by completing at a closed point, we may assume that $\mathcal{S}=\mathrm{Spf}\,A$ where $A=W(k)\llbracket T_{1},\ldots,T_{r}\rrbracket$. There is a canonical map $A\rightarrow W(A)$ mapping the $T_{i}$ to their Teichm\"{u}ller representatives in $W(A)$. The homomorphism $A\rightarrow R$ induces a map $W(A)\rightarrow W(R)$, hence $A\rightarrow R$ factors through the composition $A\rightarrow W(A)\rightarrow W(R)\rightarrow W_{n}(R)$ for each $n\in\mathbb{N}$. Then the $X_{n}:=\mathcal{X}\times_{\mathcal{S}}\mathrm{Spec}\,W_{n}(R)$ form a compatible system of smooth liftings of $X=\mathcal{X}\times_{\mathcal{S}}\mathrm{Spec}\,R$ which satisfies (A1) and (A2) by base change.
\end{rem}

\begin{rem}
Since our techniques are crystalline in nature, the smoothness of $\mathcal{S}$ seems to be indispensable in our approach. It would be interesting to understand relative Hodge-Witt decompositions for varieties with obstructed deformations. 
\end{rem}

Before giving the proof of Theorem \ref{deformation}, we shall recall the theory of ordinary Hodge $F$-crystals from \cite{Del81a}. Let $A=W(k)\llbracket T_{1},\ldots, T_{r}\rrbracket$ and $A_{0}=k\llbracket T_{1},\ldots, T_{r}\rrbracket$, for some $r\geq 0$. A \emph{crystal over $A_{0}$} is a finitely generated free $A$-module $H$ together with an integrable and topologically nilpotent connection
\begin{equation*}
\nabla\,:\,H\rightarrow H\otimes_{A}\Omega_{A/W(k)}^{1}\,.
\end{equation*}

A ring endomorphism $\phi:A\rightarrow A$ which restricts to the Frobenius on $W(k)$ is called a \emph{lift of Frobenius} if it reduces modulo $p$ to the Frobenius endomorphism $\sigma:A_{0}\rightarrow A_{0}$ which sends $x$ to $x^{p}$. A crystal $(H,\nabla)$ over $A_{0}$ is called an \emph{$F$-crystal over $A_{0}$} if for every lift of Frobenius $\phi:A\rightarrow A$, there is a given $A$-module homomorphism $F(\phi):\phi^{\ast}H\rightarrow H$ which is horizontal for $\nabla$, i.e. the square 
\begin{equation*}
\begin{tikzpicture}[descr/.style={fill=white,inner sep=1.5pt}]
        \matrix (m) [
            matrix of math nodes,
            row sep=2.5em,
            column sep=2.5em,
            text height=1.5ex, text depth=0.25ex
        ]
        { \phi^{\ast}H & \phi^{\ast}H\otimes_{A}\Omega_{A/W(k)}^{1}  \\
       H  & H\otimes_{A}\Omega_{A/W(k)}^{1} \\};

        \path[overlay,->, font=\scriptsize] 
        (m-1-1) edge node [above]{$\phi^{\ast}\nabla$} (m-1-2)
        (m-1-1) edge node [left] {$F(\phi)$}(m-2-1)
        (m-1-2) edge node [right] {$F(\phi)\otimes\mathrm{id}$} (m-2-2)
        (m-2-1) edge node [above]{$\nabla$} (m-2-2)
        ;        
                                        
\end{tikzpicture}
\end{equation*}
commutes, and such that $F(\phi)\otimes\mathbb{Q}_{p}$ is an isomorphism.  For any two liftings of Frobenius $\phi,\psi:A\rightarrow A$, we also require that $F(\psi)\circ\chi(\phi,\psi)=F(\phi)$, where $\chi(\phi,\psi):\phi^{\ast}H\xrightarrow{\sim}\psi^{\ast}H$ is the usual isomorphism of $A$-modules coming from parallel transport with respect to the connection $\nabla$:
\begin{align*}
& \chi(\phi,\psi):\phi^{\ast}H\rightarrow\psi^{\ast}H \\
& x\mapsto\sum_{m_{1},\ldots,m_{r}\geq 0}\prod_{j=1}^{r}\frac{p^{m_{j}}}{m_{j}}\left(\frac{\phi(T_{j})-\psi(T_{j})}{p}\right)^{m_{j}}\psi^{\ast}(D_{1}^{m_{1}}\circ\cdots\circ D_{n}^{m_{n}}(x)).
\end{align*}
Here the operator $D_{j}:H\rightarrow H$ denotes $(\frac{d}{dT_{j}}\otimes  1)\circ\nabla$. An $F$-crystal $(H,\nabla,F)$ is said to be a \emph{unit $F$-crystal} if $F(\phi)$ is an isomorphism for some (hence any) lift of Frobenius $\phi:A\rightarrow A$.

Let $H$ be an $F$-crystal over $A$ (we henceforth drop the $\nabla$ and $F$ from the notation), and write $H_{0}:=H\otimes_{A}A_{0}$. Given a lift of Frobenius $\phi:H\rightarrow H$, define a decreasing filtration $\mathrm{Fil}^{\bullet}H_{0}$ and an increasing filtration $\mathrm{Fil}^{\bullet}H_{0}$ of $H_{0}$ by $A_{0}$-submodules as follows:
\begin{align*}
& \mathrm{Fil}^{i}H_{0}:=\{x\in H_{0} \,:\, \exists y\in H\text{ with }y\equiv x\,\mathrm{mod}\,p\text{ and }F(\phi)\phi^{\ast}y\in p^{i}H\} \\
& \mathrm{Fil}_{i}H_{0}:=\{x\in H_{0} \,:\, \exists y\in H\text{ with }y\equiv x\,\mathrm{mod}\,p\text{ and }p^{i}y\in\mathrm{im}\,F(\phi)\}\,.
\end{align*} 
These are the \emph{Hodge} and \emph{conjugate filtrations} of $H_{0}$, respectively. It is clear that they are finite, separated and exhaustive, and that they are independent of the choice of lift of Frobenius $\phi$ \cite[\S1.3]{Del81a}. Let $\nabla_{0}:=\nabla\mod p$ be the connection on $H_{0}$ induced by $\nabla$. Then the Hodge filtration satisfies Griffiths transversality
\begin{equation*}
\nabla_{0}\mathrm{Fil}^{i}H_{0}\subset\mathrm{Fil}^{i-1}H_{0}\otimes_{A_{0}}\Omega_{A_{0}/k}^{1}
\end{equation*}
and the conjugate filtration is horizontal for $\nabla_{0}$
\begin{equation*}
\nabla_{0}\mathrm{Fil}_{i}H_{0}\subset\mathrm{Fil}_{i}H_{0}\otimes_{A_{0}}\Omega_{A_{0}/k}^{1}\,.
\end{equation*}
An $F$-crystal $H$ over $A_{0}$ is called \emph{ordinary} if the graded $A_{0}$-module $\mathrm{gr}^{\bullet}H_{0}$ associated to the Hodge filtration (equivalently the graded $A_{0}$-module $\mathrm{gr}_{\bullet}H_{0}$ associated to the conjugate filtration) is free, and the Hodge and conjugate filtrations are opposite, that is if 
\begin{equation*}
H_{0}=\mathrm{Fil}_{i}H_{0}\oplus\mathrm{Fil}^{i+1}H_{0}
\end{equation*}
for every $i$. It is shown in \cite[Prop. 1.3.2]{Del81a} that $H$ is ordinary if and only there exists a (unique) increasing filtration $U_{\bullet}$ (the \emph{conjugate filtration}) of $H$ by sub-$F$-crystals such that
\begin{equation*}
U_{i}\otimes_{A}A_{0}=\mathrm{Fil}_{i}H_{0}
\end{equation*}
for every $i$, and such that
\begin{equation*}
\mathrm{gr}_{i}H:=U_{i}/U_{i-1}\simeq V_{i}(-i)
\end{equation*}
is the $(-i)$-fold Tate twist of a unit $F$-crystal.

A \emph{Hodge $F$-crystal over $A_{0}$} is an $F$-crystal $H$ over $A_{0}$ together with a finite decreasing filtration $\mathrm{Fil}^{\bullet}H$ (the \emph{Hodge filtration}) of $H$ by free $A$-submodules which lifts the Hodge filtration on $H_{0}$, and satisfies Griffiths tranvsersality:
\begin{align*}
&\mathrm{Fil}^{i}H\otimes_{A}A_{0}=\mathrm{Fil}^{i}H_{0}\,,
&\nabla\mathrm{Fil}^{i}H\subset\mathrm{Fil}^{i-1}H\otimes_{A}\Omega_{A/W(k)}^{1}\,.
\end{align*}
A Hodge $F$-crystal $(H,\mathrm{Fil}^{\bullet}H)$ is said to be \emph{ordinary} if its underlying $F$-crystal $H$ is ordinary. It is shown in \cite[Prop. 1.3.6]{Del81a} that the conjugate and Hodge filtrations of an ordinary Hodge $F$-crystal $(H,\mathrm{Fil}^{\bullet}H)$ are opposite, that is
\begin{equation*}
H=U_{i}\oplus\mathrm{Fil}^{i+1}H
\end{equation*}
for every $i$, and one has a Hodge decomposition \cite[(1.3.6.1)]{Del81a}
\begin{equation*}
H=\bigoplus_{i}H^{i}
\end{equation*}
with $H^{i}=U_{i}\cap\mathrm{Fil}^{i}H$.
\begin{rem}\label{inclusion iso}
The induced inclusions $H^{i}\subset U_{i}/U_{i-1}$ and $H^{i}\subset\mathrm{Fil}^{i}H/\mathrm{Fil}^{i+1}H$ must then be isomorphisms, so
\begin{equation*}
H^{i}\cong U_{i}/U_{i-1}\cong\mathrm{gr}^{i}H\,.
\end{equation*} 
\end{rem}

We now prove Theorem \ref{deformation}:
\begin{proof}
After possibly taking the formal completion of $\mathcal{S}$ at a closed point, we may and do assume that $\mathcal{S}$ is $\mathcal{S}=\mathrm{Spf}\,A$ where $A=W(k)\llbracket T_{1},\ldots, T_{r}\rrbracket$. Let $\mathcal{X}_{0}:=\mathcal{X}\times_{A}A_{0}$.

The Gauss-Manin connection of the family $f:\mathcal{X}\rightarrow\mathcal{S}$ 
\begin{equation*}
\nabla:H_{\mathrm{dR}}^{s}(\mathcal{X}/A)\rightarrow H_{\mathrm{dR}}^{s}(\mathcal{X}/A)\otimes_{A}\Omega_{A/W(k)}^{1}
\end{equation*}
gives the crystalline cohomology $H_{\mathrm{cris}}^{s}(\mathcal{X}_{0}/A)\simeq H_{\mathrm{dR}}^{s}(\mathcal{X}/A)$, together with its crystalline Frobenius, the structure of an $F$-crystal over $A_{0}$. The Hodge filtration $\mathrm{Fil}^{\bullet}H_{\mathrm{dR}}^{s}(\mathcal{X}/A)$ satisfies Griffiths transversality and thus the pair $(H_{\mathrm{dR}}^{s}(\mathcal{X}/A),\mathrm{Fil}^{\bullet}H_{\mathrm{dR}}^{s}(\mathcal{X}/A))$ is a Hodge $F$-crystal over $A_{0}$. Let $e_{0}:A_{0}\rightarrow k$ denote the augmentation map. Since the closed fibre $X_{k}$ is ordinary, the Newton and Hodge polygons of $e_{0}^{\ast}H_{\mathrm{dR}}^{s}(\mathcal{X}/A))\simeq H_{\mathrm{cris}}^{s}(X_{k}/W(k))$ coincide. Therefore $(H_{\mathrm{dR}}^{s}(\mathcal{X}/A),\mathrm{Fil}^{\bullet}H_{\mathrm{dR}}^{s}(\mathcal{X}/A))$ is an ordinary Hodge $F$-crystal by \cite[Prop. 1.3.2]{Del81a}, and hence we have the conjugate filtration $U_{\bullet}$ of $H_{\mathrm{cris}}^{s}(\mathcal{X}_{0}/A)$ lifting the conjugate filtration $\mathrm{Fil}_{\bullet}H_{\mathrm{cris}}^{s}(\mathcal{X}_{0}/A_{0})$, and such that $U_{i}/U_{i-1}\simeq V_{i}(-i)$ is the $(-i)$-fold Tate twist of a unit $F$-crystal. Moreover, the filtration $U_{\bullet}$ is opposite to the Hodge filtration $\mathrm{Fil}^{\bullet}H_{\mathrm{dR}}^{s}(\mathcal{X}/A)$.

Now we evaluate this filtration of $F$-crystals on $W(A)$ and get a filtration of $W(A)$-modules 
\begin{equation*}
0\subset {U_{0}}_{W(A)}\subset {U_{1}}_{W(A)}\subset\ldots\subset {U_{n}}_{W(A)}=H_{\rm{cris}}^{s}(\mathcal{X}_{0}/W(A))\,.
\end{equation*}
Note that the Frobenius on $W(A)$ is a lifting of the Frobenius on $A_{0}=W(A)/\langle\tensor[^V]{W(A)}{},p\rangle$. Then ${U_{i}}_{W(A)}/{U_{i-1}}_{W(A)}$ is a free $W(A)$-module such that the crystalline Frobenius induces the $(-i)$-fold Tate twist of the multiplicative (unit-) crystal evaluated at $W(A)$ on which $F$ acts as an $F$-linear isomorphism. If $U$ is a unit $F$-crystal over $A_{0}$ and we evaluate it at $W(A)$, then $F:W(A)\otimes_{\sigma,W(A)}U_{W(A)}\rightarrow U_{W(A)}$ is an isomorphism. By definition,
\begin{equation*}
(P_{0}=U_{W(A)},P_{1}=\tensor[^V]{W(A)}{}P_{0},F,F_{1}:P_{1}\rightarrow P_{0})
\end{equation*}
with $F_{1}:\tensor[^V]{\xi}{}\alpha\mapsto\xi F\alpha$ is then a multiplicative display. If $U$ is the Tate-twist of a unit $F$-crystal then $U_{W(A)}$ is the Tate-twist of a multiplicative display. We can now take the base change of $F$-crystals, respectively of displays, with respect to the map $A\rightarrow R$ to get a filtration
\begin{equation*}
0\subset {U_{0}}_{W(R)}\subset {U_{1}}_{W(R)}\subset\ldots\subset {U_{n}}_{W(R)}=H_{\rm{cris}}^{s}(X/W(R))
\end{equation*}
of $F$-crystals evaluated at $W(R)$, such that the successive quotients are Tate-twists of multiplicative displays.

We already know that (by Remark \ref{inclusion iso})
\begin{equation*}
{U_{0}}_{W(R)}\simeq H^{s}(\mathfrak{X},\mathcal{O}_{\mathfrak{X}})
\end{equation*}
and we claim that the composite map 
\begin{equation*}
\varsigma:{U_{0}}_{W(R)}\rightarrow H_{\mathrm{cris}}^{s}(X/W(R))\rightarrow H^{s}(X,W\mathcal{O}_{X})
\end{equation*}
is an isomorphism. Note that the diagram 
\begin{equation*}
\begin{tikzpicture}[descr/.style={fill=white,inner sep=1.5pt}]
        \matrix (m) [
            matrix of math nodes,
            row sep=4em,
            column sep=2em,
            text height=1.5ex, text depth=0.25ex
        ]
        { {U_{0}}_{W(R)} & H^{s}(X,W\mathcal{O}_{X})  \\
       {U_{0}}_{W(R)} & H^{s}(X,W\mathcal{O}_{X}) \\};

        \path[overlay,->, font=\scriptsize] 
        (m-1-1) edge node [above]{$\varsigma$} (m-1-2)
        (m-2-1) edge node [above]{$\varsigma$} (m-2-2)
        (m-1-1) edge node [left]{$F$} (m-2-1)
        (m-1-2) edge node [right]{$F$} (m-2-2)
        ;
                                
\end{tikzpicture}
\end{equation*}
commutes by construction. To see the claim, let $\mathfrak{X}$ denote the ind-scheme over $\mathrm{Spec}\,W_{\bullet}(R)$ arising from the compatible family of liftings $X_{n}/\mathrm{Spec}\,W_{n}(R)$ constructed using the canonical map $A\rightarrow W(A)\rightarrow W_{n}(R)$ as in Remark \ref{assumptions}. Then for each $i$ we have maps
\begin{equation*}
H^{i}_{\mathrm{cris}}(X/W(R))\cong H_{\mathrm{dR}}^{i}(\mathfrak{X}/W(R))\twoheadrightarrow H^{i}(\mathfrak{X},\mathcal{O}_{\mathfrak{X}})\twoheadrightarrow H^{i}(X,\mathcal{O}_{X})
\end{equation*}  
where the final arrow is surjective because de Rham cohomology and Hodge cohomology commute with base change. The composition factors through the map $H^{i}(X,W\mathcal{O}_{X})\rightarrow H^{i}(X,\mathcal{O}_{X})$ induced by $W\mathcal{O}_{X}\rightarrow\mathcal{O}_{X}$. Hence the cohomology sequence coming from the short exact sequence 
\begin{equation*}
0\rightarrow W\mathcal{O}_{X}\xrightarrow{V}W\mathcal{O}_{X}\rightarrow\mathcal{O}_{X}\rightarrow 0
\end{equation*}
splits into short exact sequences
\begin{equation*}
0\rightarrow H^{i}(X,W\mathcal{O}_{X})\xrightarrow{V}H^{i}(X,W\mathcal{O}_{X})\rightarrow H^{i}(X,\mathcal{O}_{X})\rightarrow 0
\end{equation*}
for each $i$. In particular, we see that the map $\varsigma$ reduces modulo $I_{R}$ to the map
\begin{equation*}
\overline{\varsigma}:{U_{0}}_{W(R)}/I_{R}{U_{0}}_{W(R)}\rightarrow H^{s}(X,W\mathcal{O}_{X})/\mathrm{im}\,V\simeq H^{s}(X,\mathcal{O}_{X})\,.
\end{equation*}
As in the previous sections, we see that $\overline{\varsigma}$ is an isomorphism by reducing to the case $R=k$, where $H^{s}(X_{k},W\mathcal{O}_{X_{k}})$ is the slope $0$ part in $H^{s}_{\mathrm{cris}}(X_{k}/W(k))$. By Lemma \ref{(iii)} we conclude that $\varsigma$ is an isomorphism, hence $\zeta$ imposes a multiplicative display structure on $H^{s}(X,W\mathcal{O}_{X})$. It is clear that under the isomorphism $U_{0}\cong H^{s}(\mathfrak{X},\mathcal{O}_{\mathfrak{X}})$, the composite map
\begin{equation*}
H^{s}(\mathfrak{X},\mathcal{O}_{\mathfrak{X}})\rightarrow H^{s}_{\mathrm{dR}}(\mathfrak{X}/W(R))\simeq H_{\mathrm{cris}}^{s}(X/W(R))\rightarrow H^{s}(X,W\mathcal{O}_{X})
\end{equation*}
agrees with the map $\alpha_{s}$.

Since ${U_{0}}_{W(R)}$ is a direct summand of $H_{\rm{cris}}^{s}(X/W(R))$ we conclude that ${U_{1}}_{W(R)}/{U_{0}}_{W(R)}$ is a direct summand too. Since $H^{s}(X,W\mathcal{O}_{X})$ is a direct summand of $H^{s}_{\mathrm{cris}}(X/W(R))$, we get an induced map
\begin{equation*}
\begin{adjustbox}{width=12.5cm}
$\varsigma:P_{1}={U_{1}}_{W(R)}/{U_{0}}_{W(R)}\rightarrow\mathbb{H}^{s}(X,0\rightarrow W\Omega_{X/R}^{1}\xrightarrow{d}\cdots\xrightarrow{d}W\Omega_{X/R}^{s})\rightarrow H^{s-1}(X,W\Omega_{X/R}^{1})\,.$
\end{adjustbox}
\end{equation*}
The diagram
\begin{equation*}
\begin{tikzpicture}[descr/.style={fill=white,inner sep=1.5pt}]
        \matrix (m) [
            matrix of math nodes,
            row sep=4em,
            column sep=2em,
            text height=1.5ex, text depth=0.25ex
        ]
        { P_{1} & H^{s-1}(X,W\Omega_{X/R}^{1})  \\
       P_{1} & H^{s-1}(X,W\Omega_{X/R}^{1}) \\};

        \path[overlay,->, font=\scriptsize] 
        (m-1-1) edge node [above]{$\varsigma$} (m-1-2)
        (m-2-1) edge node [above]{$\varsigma$} (m-2-2)
        (m-1-1) edge node [left]{$F$} (m-2-1)
        (m-1-2) edge node [right]{$F$} (m-2-2)
        ;
                                
\end{tikzpicture}
\end{equation*}
is commutative because it already commutes when $X$ is replaced with the versal deformation $\mathcal{X}$. The $F$ on the left, which is the untwist of the crystalline Frobenius, is an $F$-linear isomorphism because it is defined on a multiplicative (unit-) display.

We claim that we have exact sequences 
\begin{equation*}
0\rightarrow H^{i}(X,W\Omega_{X/R}^{1})\xrightarrow{V}H^{i}(X,W\Omega_{X/R}^{1})\rightarrow H^{i}(X,\Omega_{X/R}^{1})\rightarrow 0
\end{equation*}
induced by the action of $V$ on $W\Omega_{X/R}^{1}$, for all $i$. Indeed, consider the commutative diagram
\begin{equation*}
\begin{tikzpicture}[descr/.style={fill=white,inner sep=1.5pt}]
        \matrix (m) [
            matrix of math nodes,
            row sep=4em,
            column sep=1.8em,
            text height=1.5ex, text depth=0.25ex
        ]
        { 0 & \mathbb{H}^{i}(X,N^{2}W\Omega_{X/R}^{\geq 1}) & \mathbb{H}^{i+1}(X,N^{2}W\Omega_{X/R}^{\bullet}) &  H^{i+1}(X,W\mathcal{O}_{X}) & 0 \\
       \ & \ & \mathbb{H}^{i+1}(\mathfrak{X},\mathcal{F}^{2}\Omega_{\mathfrak{X}/W(R)}^{\bullet}) &  I_{R}H^{i+1}(\mathfrak{X},\mathcal{O}_{\mathfrak{X}}) & \ \\};

        \path[overlay,->, font=\scriptsize] 
        (m-1-1) edge (m-1-2)
        (m-1-2) edge (m-1-3)
        (m-1-3) edge (m-1-4)
        (m-1-4) edge (m-1-5)
        (m-2-3) edge (m-2-4)
        (m-2-3) edge node[right]{$\cong$} (m-1-3)
        (m-2-4) edge node[left]{$\cong$} node[right]{$\tensor[^F]{\alpha}{_0}$} (m-1-4);
                                
\end{tikzpicture}
\end{equation*}
The last map is surjective because the $E_{1}$-hypercohomology spectral sequence associated to $\mathcal{F}^{2}\Omega_{\mathfrak{X}/W(R)}^{\bullet}$ degenerates. Applying the same argument in cohomological degree $i$, we see that the first map is injective. Therefore
\begin{equation*}
\mathbb{H}^{i}(X,N^{2}W\Omega_{X/R}^{\geq 1})\cong \mathbb{H}^{i}(\mathfrak{X},\mathcal{F}^{2}\Omega_{\mathfrak{X}/W(R)}^{\geq 1})\,.
\end{equation*}
Hence the composite map
\begin{equation*}
\mathbb{H}^{i}(X,N^{2}W\Omega_{X/R}^{\geq 1})\xrightarrow{(V,\mathrm{id})}\mathbb{H}^{i}(X,W\Omega_{X/R}^{\geq 1})\rightarrow\mathbb{H}^{i}(X,W\Omega_{X/R}^{\bullet})
\end{equation*}
can be identified under the comparison between de Rham-Witt and de Rham cohomology with the composite map (see \eqref{comm square 2})
\begin{equation*}
\mathbb{H}^{i}(\mathfrak{X},\mathcal{F}^{2}W\Omega_{\mathfrak{X}/W(R)}^{\geq 1})\rightarrow\mathbb{H}^{i}(\mathfrak{X},\Omega_{\mathfrak{X}/W(R)}^{\geq 1})\rightarrow H^{i}_{\mathrm{dR}}(\mathfrak{X}/W(R))
\end{equation*}
which is injective due to the degeneracy of the $E_{1}$-hypercohomology spectral sequences. We conclude that the map
\begin{equation*}
\mathbb{H}^{i}(X,N^{2}W\Omega_{X/R}^{\geq 1})\xrightarrow{(V,\mathrm{id})}\mathbb{H}^{i}(X,W\Omega_{X/R}^{\geq 1})
\end{equation*}
is injective and has cokernel $H^{i}(X,\Omega_{X/R}^{1})$. Now consider the diagram
\begin{equation*}
\begin{tikzpicture}[descr/.style={fill=white,inner sep=1.5pt}]
        \matrix (m) [
            matrix of math nodes,
            row sep=4em,
            column sep=2em,
            text height=1.5ex, text depth=0.25ex
        ]
        { \mathbb{H}^{i}(W\Omega^{\geq 2}) & \mathbb{H}^{i}(W\Omega^{\geq 1}) & H^{i}(W\Omega^{1}) & \mathbb{H}^{i+1}(W\Omega^{\geq 2}) & \ \\
      \mathbb{H}^{i}(W\Omega^{\geq 2}) & \mathbb{H}^{i}(N^{2}W\Omega^{\geq 1}) & H^{i}(W\Omega^{1}) & \mathbb{H}^{i+1}(W\Omega^{\geq 2}) & \  \\
       };

        \path[overlay,->, font=\scriptsize] 
        (m-2-1) edge node [left]{$=$} node [right]{$\mathrm{id}$} (m-1-1)
        (m-2-3) edge node [right] {$V$} (m-1-3)
        (m-2-2) edge node[right]{$(V,\mathrm{id})$} (m-1-2)
        (m-1-2) edge (m-1-3)
        (m-2-2) edge (m-2-3)
        (m-1-3) edge (m-1-4)
        (m-2-3) edge (m-2-4)
        (m-2-4) edge node [right]{$\mathrm{id}$} node [left]{$=$} (m-1-4)
        (m-1-4) edge (m-1-5)
        (m-2-4) edge (m-2-5)
        (m-1-1) edge (m-1-2)
        (m-2-1) edge (m-2-2);
                        
\end{tikzpicture}
\end{equation*}
The injectivity of $(V,\mathrm{id})$ in cohomological degrees $i$ and $i+1$ implies that $V$ is injective on $H^{i}(X,W\Omega_{X/R}^{1})$ and has cokernel $H^{i}(X,\Omega_{X/R}^{1})$ as desired.

As before, we see that the induced map 
\begin{equation*}
\overline{\varsigma}:P_{1}/I_{R}P_{1}\rightarrow H^{s-1}(X,W\Omega_{X/R}^{1})/\mathrm{im}\, V\simeq H^{s-1}(X,\Omega_{X/R}^{1})
\end{equation*}
is an isomorphism by reducing to the case $R=k$, where $H^{s-1}(X_{k},W\Omega_{X_{k}/k}^{1})$ is the slope $1$ part in $H_{\mathrm{cris}}^{s}(X_{k}/W(k))$. By Lemma \ref{(iii)} we conclude that $\varsigma$ is an isomorphism and that $P_{1}\cong H^{s-1}(\mathfrak{X},\Omega_{\mathfrak{X}/W(R)}^{1})$ (see Remark \ref{inclusion iso}). It is clear that the map $\zeta$ agrees with the map
\begin{equation*}
\alpha_1:H^{s-1}(\mathfrak{X},\Omega_{\mathfrak{X}/W(R)}^{1})\rightarrow H^{s-1}(X,W\Omega_{X/R}
^{1})
\end{equation*} 
and 
\begin{align*}
\tensor[^F]{\alpha}{_1}: \ 
& I_{R}H^{s-1}(\mathfrak{X},\Omega_{\mathfrak{X}/W(R)}^{1})\rightarrow H^{s-1}(X,W\Omega_{X/R}
^{1}) \\
& \ \ \ \ \ \ \ z\mapsto F(\alpha_{1}(z))
\end{align*}
is a bijection.

Now consider the composite map 
\begin{equation*}
\begin{adjustbox}{width=12.5cm}
$\varsigma:P_{2}={U_{2}}_{W(R)}/{U_{1}}_{W(R)}\rightarrow\mathbb{H}^{s}(X,0\rightarrow 0\rightarrow W\Omega_{X/R}^{2}\xrightarrow{d}\cdots\xrightarrow{d}W\Omega_{X/R}^{s})\rightarrow H^{s-2}(X,W\Omega_{X/R}^{2})\,.$
\end{adjustbox}
\end{equation*}
The same argument as before shows that the square
\begin{equation*}
\begin{tikzpicture}[descr/.style={fill=white,inner sep=1.5pt}]
        \matrix (m) [
            matrix of math nodes,
            row sep=4em,
            column sep=2em,
            text height=1.5ex, text depth=0.25ex
        ]
        { P_{2} & H^{s-2}(X,W\Omega_{X/R}^{2})  \\
       P_{2} & H^{s-2}(X,W\Omega_{X/R}^{2}) \\};

        \path[overlay,->, font=\scriptsize] 
        (m-1-1) edge node [above]{$\varsigma$} (m-1-2)
        (m-2-1) edge node [above]{$\varsigma$} (m-2-2)
        (m-1-1) edge node [left]{$F$} (m-2-1)
        (m-1-2) edge node [right]{$F$} (m-2-2)
        ;
                                
\end{tikzpicture}
\end{equation*}
is commutative, where $F$ on the left is again the untwist of the crystalline Frobenius.

We claim that we have exact sequences 
\begin{equation*}
0\rightarrow H^{i}(X,W\Omega_{X/R}^{2})\xrightarrow{V}H^{i}(X,W\Omega_{X/R}^{2})\rightarrow H^{i}(X,\Omega_{X/R}^{2})\rightarrow 0
\end{equation*}
induced by the action of $V$ on $W\Omega_{X/R}^{2}$, for all $i$. We proceed in a similar manner as before: Consider the commutative diagram
\begin{equation*}
\begin{tikzpicture}[descr/.style={fill=white,inner sep=1.5pt}]
        \matrix (m) [
            matrix of math nodes,
            row sep=4em,
            column sep=1.8em,
            text height=1.5ex, text depth=0.25ex
        ]
        { 0 & \mathbb{H}^{i}(X,N^{3}W\Omega_{X/R}^{\geq 1}) & \mathbb{H}^{i+1}(X,N^{3}W\Omega_{X/R}^{\bullet}) &  H^{i+1}(X,W\mathcal{O}_{X}) & 0 \\
       \ & \ & \mathbb{H}^{i+1}(\mathfrak{X},\mathcal{F}^{3}\Omega_{\mathfrak{X}/W(R)}^{\bullet}) &  I_{R}H^{i+1}(\mathfrak{X},\mathcal{O}_{\mathfrak{X}}) & \ \\};

        \path[overlay,->, font=\scriptsize] 
        (m-1-1) edge (m-1-2)
        (m-1-2) edge (m-1-3)
        (m-1-3) edge (m-1-4)
        (m-1-4) edge (m-1-5)
        (m-2-3) edge (m-2-4)
        (m-2-3) edge node[right]{$\cong$} (m-1-3)
        (m-2-4) edge node[left]{$\cong$} node[right]{$\tensor[^F]{\alpha}{_0}$} (m-1-4);
                                
\end{tikzpicture}
\end{equation*}
The last map is surjective because the $E_{1}$-hypercohomology spectral sequence associated to $\mathcal{F}^{3}\Omega_{\mathfrak{X}/W(R)}^{\bullet}$ degenerates, and the same argument in cohomological degree $i$, shows the first map is injective. Therefore
\begin{equation*}
\mathbb{H}^{i}(X,N^{3}W\Omega_{X/R}^{\geq 1})\cong \mathbb{H}^{i}(\mathfrak{X},\mathcal{F}^{3}\Omega_{\mathfrak{X}/W(R)}^{\geq 1})
\end{equation*}
for each $i$.

Considering the commutative diagram 
\begin{equation*}
\begin{tikzpicture}[descr/.style={fill=white,inner sep=1.5pt}]
        \matrix (m) [
            matrix of math nodes,
            row sep=4em,
            column sep=1.8em,
            text height=1.5ex, text depth=0.25ex
        ]
        { 0 & \mathbb{H}^{i}(X,N^{3}W\Omega_{X/R}^{\geq 2}) & \mathbb{H}^{i+1}(X,N^{3}W\Omega_{X/R}^{\geq 1}) &  H^{i+1}(X,W\Omega_{X/R}^{1}) & 0 \\
       \ & \ & \mathbb{H}^{i+1}(\mathfrak{X},\mathcal{F}^{3}\Omega_{\mathfrak{X}/W(R)}^{\geq 1}) &  I_{R}H^{i+1}(\mathfrak{X},\Omega_{\mathfrak{X}/W(R)}^{1}) & \ \\};

        \path[overlay,->, font=\scriptsize] 
        (m-1-1) edge (m-1-2)
        (m-1-2) edge (m-1-3)
        (m-1-3) edge (m-1-4)
        (m-1-4) edge (m-1-5)
        (m-2-3) edge (m-2-4)
        (m-2-3) edge node[right]{$\cong$} (m-1-3)
        (m-2-4) edge node[left]{$\cong$} node[right]{$\tensor[^F]{\alpha}{_1}$} (m-1-4);
                                
\end{tikzpicture}
\end{equation*}
then shows that
\begin{equation*}
\mathbb{H}^{i}(X,N^{3}W\Omega_{X/R}^{\geq 2})\cong \mathbb{H}^{i}(\mathfrak{X},\mathcal{F}^{3}\Omega_{\mathfrak{X}/W(R)}^{\geq 2})
\end{equation*}
for each $i$, as well. Finally, the diagram
\begin{equation*}
\begin{tikzpicture}[descr/.style={fill=white,inner sep=1.5pt}]
        \matrix (m) [
            matrix of math nodes,
            row sep=4em,
            column sep=2em,
            text height=1.5ex, text depth=0.25ex
        ]
        { \mathbb{H}^{i}(W\Omega^{\geq 3}) & \mathbb{H}^{i}(W\Omega^{\geq 2}) & H^{i}(W\Omega^{2}) & \mathbb{H}^{i+1}(W\Omega^{\geq 3}) & \ \\
      \mathbb{H}^{i}(W\Omega^{\geq 3}) & \mathbb{H}^{i}(N^{3}W\Omega^{\geq 2}) & H^{i}(W\Omega^{2}) & \mathbb{H}^{i+1}(W\Omega^{\geq 3}) & \  \\
       };

        \path[overlay,->, font=\scriptsize] 
        (m-2-1) edge node [left]{$=$} node [right]{$\mathrm{id}$} (m-1-1)
        (m-2-3) edge node [right] {$V$} (m-1-3)
        (m-2-2) edge node[right]{$(V,\mathrm{id})$} (m-1-2)
        (m-1-2) edge (m-1-3)
        (m-2-2) edge (m-2-3)
        (m-1-3) edge (m-1-4)
        (m-2-3) edge (m-2-4)
        (m-2-4) edge node [right]{$\mathrm{id}$} node [left]{$=$} (m-1-4)
        (m-1-4) edge (m-1-5)
        (m-2-4) edge (m-2-5)
        (m-1-1) edge (m-1-2)
        (m-2-1) edge (m-2-2);
                        
\end{tikzpicture}
\end{equation*}
and injectivity of $(V,\mathrm{id})$ in cohomological degrees $i$ and $i+1$ (compare \eqref{comm square 2}) implies that $V$ is injective on $H^{i}(X,W\Omega_{X/R}^{2})$ and has cokernel $H^{i}(X,\Omega_{X/R}^{2})$ as desired.

By the same argument as before, we see that the induced map 
\begin{equation*}
\overline{\varsigma}:P_{2}/I_{R}P_{2}\rightarrow H^{s-2}(X,W\Omega_{X/R}^{2})/\mathrm{im}\, V\simeq H^{s-2}(X,\Omega_{X/R}^{2})
\end{equation*}
is an isomorphism by reducing to the case $R=k$, where $H^{s-2}(X_{k},W\Omega_{X_{k}/k}^{2})$ is the slope $2$ part in $H_{\mathrm{cris}}^{s}(X_{k}/W(k))$. By Lemma \ref{(iii)} we conclude that $\varsigma$ is an isomorphism and that $P_{2}\cong H^{s-2}(\mathfrak{X},\Omega_{\mathfrak{X}/W(R)}^{2})$ (by Remark \ref{inclusion iso}). 

By an induction argument as in the case of abelian schemes one derives an exact sequence for all $i$ and $j$
\begin{equation*}
0\rightarrow H^{i}(X,W\Omega_{X/R}^{j})\xrightarrow{V}H^{i}(X,W\Omega_{X/R}^{j})\rightarrow H^{i}(X,\Omega_{X/R}^{j})\rightarrow 0
\end{equation*} 
and that the map
\begin{equation*}
P_{i}:={U_{i}}_{W(R)}/{U_{i-1}}_{W(R)}\simeq H^{s-i}(\mathfrak{X},\Omega_{\mathfrak{X}}^{i})\xrightarrow{\varsigma} H^{s-i}(X,W\Omega_{X/R}^{i})
\end{equation*}
is an isomorphism, and hence imposes a multiplicative display structure on $H^{s-i}(X,W\Omega_{X/R}^{i})$ and identifies this Hodge-Witt cohomology group as a direct summand in crystalline cohomology. This concludes the proof of Theorem \ref{deformation}.
\end{proof}

\end{document}